\documentclass[10pt,a4paper,twoside]{amsart}
\usepackage{amsmath,amsfonts,amsthm,amsopn,color,amssymb,enumitem}
\usepackage{palatino}
\usepackage{graphicx}
\usepackage[colorlinks=true]{hyperref}
\hypersetup{urlcolor=blue, citecolor=red, linkcolor=blue}

\newcommand{\e}{\varepsilon}
\newcommand{\eps}{\varepsilon}

\newcommand{\R}{\mathbb{R}}

\newcommand{\RN}{{\mathbb{R}^N}}

\newcommand{\de}{\partial}

\DeclareMathOperator{\dist}{dist}

 \DeclareMathOperator{\Id}{Id}
\renewcommand{\le}{\leslant}
\renewcommand{\ge}{\geslant}
\renewcommand{\a }{\alpha }

\renewcommand{\b }{\beta }

\renewcommand{\d }{\delta }

\newcommand{\g }{\gamma }

\renewcommand{\l }{\lambda}
\newcommand{\n }{\nabla }

\renewcommand{\t}{\theta}

\newcommand{\G}{\Gamma}

\renewcommand{\H}{H^1(\RN)}
\newcommand{\Hr}{H^1_r(\RN)}

\newcommand{\N}{\mathbb{N}}

\newcommand{\irn }{\int_{\RN}}

\def\bbm[#1]{\mbox{\boldmath $#1$}}
\newcommand{\beq }{\begin{equation}}
\newcommand{\eeq }{\end{equation}}

\newcommand{\Ce}{\mathcal{C}_\eps}

\renewcommand{\le}{\leqslant}
\renewcommand{\ge}{\geqslant}

\newtheorem{theorem}{Theorem}[section]
\newtheorem{lemma}[theorem]{Lemma}

\newtheorem{proposition}[theorem]{Proposition}

\title[Semi-classical states for the NLSE]{Semi-classical states for the Nonlinear Schr{\"o}dinger Equation on saddle points of the potential via variational methods}

\author[d'Avenia]{Pietro d'Avenia$^1$}
\author[Pomponio]{Alessio Pomponio$^1$}
\address{$^1$Dipartimento di Matematica, Politecnico di
Bari, Via E. Orabona 4, I-70125 Bari, Italy.}
\author[Ruiz]{David Ruiz$^2$}
\address{$^2$Dpto. An{\'a}lisis Matem{\'a}tico, Granada, 18071 Spain.}

\thanks{P. D. and A. P. are supported by M.I.U.R. - P.R.I.N.
``Metodi variazionali e topologici nello studio di fenomeni non
lineari'' and by GNAMPA Project ``Problemi ellittici con termini
non locali''. D.R has been supported by the Spanish Ministry of
Science and Innovation under Grant MTM2008-00988 and by J.
Andaluc\'{\i}a (FQM 116). Moreover, the three authors have been
supported by the Spanish-Italian Acci{\'o}n Integrada
HI2008.0106/Azione Integrata Italia-Spagna IT09L719F1.}

\email{p.davenia@poliba.it, a.pomponio@poliba.it, daruiz@ugr.es}
\date{}



\begin{document}

%\begin{abstract}
%
%\end{abstract}

\maketitle

\begin{abstract}
In this paper we study semiclassical states for the problem
$$ -\eps^2 \Delta u + V(x) u = f(u) \qquad \hbox{in } \RN,$$
where $f(u)$ is a superlinear nonlinear term. Under our hypotheses
on $f$ a Lyapunov-Schmidt reduction is not possible. We use
variational methods to prove the existence of spikes around saddle
points of the potential $V(x)$.
\end{abstract}

\section{Introduction}

%\begin{remark}
%
%Under the Ambrosetti-Rabinowitz condition, it is easy to check
%that $\lim_{x \to + \infty} f(x)/x =+\infty$. Therefore, for any
%$\eps>0$, we can define:
%
%$$ \alpha_{\eps} = \min \left \{ f(x)/x:\ x \ge \eps \right \} = f(a_{\eps})/a_{\eps}, \ a_{\eps} \ge \eps. $$
%
%Moreover, it is easy to show that $\alpha_{\eps} \to 0$ and
%$a_{\eps} \to 0$ as $\eps \to 0$. This value $a_{\eps}$ will be
%used to truncate the function $f(u)$. This technicality is needed
%to guarantee that $\tilde{f}(u) \le f(u)$ for any $u>0$.
%
%
%
%
%\end{remark}

Our starting point is the equation of the standing waves for the
Nonlinear Schr{\"o}dinger Equation:
\begin{equation}
\label{eq:Sch} -\eps^2 \Delta u + V(x) u = f(u) \qquad \hbox{in }
\RN.
\end{equation}
Here $u \in H^1(\R^N)$, $N \ge 2$, $V(x)$ is a positive potential
and $f$ is a nonlinear term. This problem has been largely studied
in the literature, and it is not possible to give here a complete
bibliography.

The existence of solutions for \eqref{eq:Sch} has been treated in
\cite{beres-lions, strauss} for constant potentials and
\cite{bahri-lions, benci-cerami, ding-ni, rabi} in more general
cases. An interesting issue concerning \eqref{eq:Sch} is the
existence of semiclassical states, which implies the study of
\eqref{eq:Sch} for small $\e>0$. From the point of view of
Physics, semiclassical states describe a kind of transition from
Quantum Mechanics to Newtonian Mechanics. In this framework one is
interested not only in existence of solutions but also in their
asymptotic behavior as $\e \to 0$.  Typically, solutions tend to
concentrate around critical points of $V$: such solutions are
called \emph{spikes}.

The first result in this direction was given by Floer and
Weinstein in \cite{fw}, where the case $N=1$ and $f(u)=u^3$ is
considered. Later, Oh generalized this result to higher values of
$N$ and $f(u)=u^p$, $1<p<\frac{N+2}{N-2}$, see \cite{oh1, oh2}. In
those papers existence of \emph{spikes} around any non-degenerate
critical point $x_0$ of $V(x)$ is proved. Roughly speaking, a
spike is a solution $u_{\e}$ such that:
$$ u_{\e} \sim U \left (\frac{x-x_0}{\e} \right ) \ \mbox{ as }
\e \to 0,$$ where $U$ is a ground state solution of the limit
problem:
\begin{equation} \label{eq:limit} - \Delta U + V(x_0) U = f(U).\end{equation}

Let us point out here that not any critical point of $V(x)$ will
generate a spike around it: for instance, it has been proved in
\cite{esteban, esteban2} that \eqref{eq:Sch} has no non-trivial
solution if $V(x)$ is decreasing along a direction (and different
from constant). However, \cite{abc, li} extended the previous
result to some possibly degenerate critical points of $V$.

All those results (\cite{abc, fw, li, oh1, oh2}) use the following
non-degeneracy condition for \eqref{eq:limit}:

\begin{enumerate}[label=(ND), ref=(ND)]

\item \label{ND} The vector space of solutions of $- \Delta w +
V(x_0) w = f'(U)w$ is generated by $\{ \partial_{x_i} U,\ i=1
\dots N.\}$.
\end{enumerate}

\medskip

This property is essential in their approach since they use a
Lyapunov-Schmidt reduction which is based on the study of the
linearized problem. The argument of the proof of \ref{ND} (see for
instance \cite{am-book}, Chapter 4) needs a non-existence result
for ODE's that has been proved only for specific types of
nonlinearities, like powers (see \cite{kwong}).

A first attempt to generalize such result without assuming
\ref{ND} was given in \cite{dpf1} (see also \cite{gui}), which was
later improved by \cite{dpf2, dpf3}. Here the procedure is
completely different, and uses a variational approach applied to a
truncated problem. In those papers the following hypotheses are
made on $f$:

\begin{enumerate}[label=(f\arabic*), ref=(f\arabic*)]
\setcounter{enumi}{-1} \item \label{f0} $f:[0, +\infty) \to \R$ is
$C^1$;

\item \label{f1} $f(s)=o(s)$ as $s \sim 0$;

\item \label{f2} $ \lim_{s\to +\infty} \frac{f(s)}{s^p}=0$ for
some $p \in (1, \frac{N+2}{N-2})$ if $N\ge 3$, or just $p >1$ if
$N=2$;

\item \label{f3} there exists $\mu>2$ such that, for every $s>0$,
\[
0<\mu F(s) < sf(s),
\]
where $F(s)=\int_0^s f(t) dt$;

\item \label{f4} the map $t \mapsto \frac{f(t)}{t}$ is
non-decreasing.
\end{enumerate}

The first two conditions imply that $f$ is superlinear and
sub-critical, and are quite natural in this framework. Condition
\ref{f3} is the so-called Ambrosetti-Rabinowitz condition, which
has been imposed many times in order to deal with superlinear
problems. Finally, condition \ref{f4} is suitable for using a
Nehari manifold approach.

Under those conditions, \cite{dpf3} shows the existence of spikes
around critical points of $V(x)$ under certain conditions. Roughly
speaking, the critical points considered are those that can be
found through a local min-max approach; this is a very general
assumption and includes of course any non-degenerate critical
point.

Recently, some papers have tried to eliminate some of the
conditions \ref{f3}-\ref{f4}, or to substitute them with other
assumptions. For instance, in \cite{bw, jt-cv} condition \ref{f4}
is removed (moreover, \cite{jt-cv} deals also with asymptotically
linear problems, where \ref{f3} is replaced with another
condition). In \cite{avila-j, jb-arma, jbt-cpde} both conditions
\ref{f3} and \ref{f4} are eliminated, and the authors assume the
minimal hypotheses under which one can prove the existence of
solution for \eqref{eq:limit} (those of \cite{beres-lions}).
However, in \cite{avila-j, jb-arma, jbt-cpde, bw, jt-cv} only the
case of local minima of $V(x)$ is considered.

The goal of this paper is to prove existence of spikes around
saddle points or maxima of $V(x)$ without assumption \ref{f4}. Our
approach is reminiscent of \cite{dpf3}; basically, we define a
conveniently modified energy functional and try to prove existence
of solution by variational methods. The main difference with
respect to \cite{dpf3} is that, since \ref{f4} is not assumed, the
Nehari manifold technique is not applicable here. So, we need to
construct a different min-max argument, which involves suitable
deformations of certain cones in $H^1(\R^N)$. This approach seems
very natural but has not been used before in the related
literature. As a second novelty, a classical property of the
Brouwer degree regarding the existence of connected sets of
solutions reveals crucial to estimate the critical values (see
\cite{leray-s, mawhin}). Indeed, this property allows us to relate
our min-max value to another min-max value with the constraint of
having center of mass equal to $0$ (see Section \ref{The min-max
argument} for a more detailed exposition).

Finally, once a solution is obtained, asymptotic estimates are
needed in order to prove that the solution of the modified problem
solves \eqref{eq:Sch}.

We assume that $V:\RN \to \R$ is a function satisfying the
following boundedness condition:

\begin{enumerate}[label=(V0), ref=(V0)]

\item \label{V0} $0<\a_1 \le V(x) \le \a_2$, for all $x \in \RN$;

\end{enumerate}

Moreover, with respect to the critical point $0$, we assume that
one of the following conditions is satisfied:

\begin{enumerate}[label=(V\arabic*), ref=(V\arabic*)]

\item \label{V1} $V(0)=1$, $V$ is $C^1$ in a neighborhood of $0$
and $0$ is an isolated local maximum of $V$.

\item \label{V2} $V(0)=1$, $V$ is $C^2$ in a neighborhood of $0$
and $0$ is a non-degenerate saddle critical point of $V$.

\item \label{V3} $V(0)=1$, $V$ is $C^{N-1}$ in a neighborhood of
$0$, $0$ is an isolated critical point of $V(x)$ and there exists
a vector space $E$ such that:

\begin{enumerate}[label=\alph*), ref=\alph*)]

\item $V|_E$ has a local maximum at $0$;

\item $V|_{E^\perp}$ has a local minimum at $0$.

\end{enumerate}

\end{enumerate}

Our assumptions on the critical points of $V$ are not as general
as in \cite{dpf3}, but still include non-degenerate cases, as well
as isolated maxima and many degenerate cases.

Our main theorem is the following:

\begin{theorem} \label{teo} Assume that $f$ satisfies hypotheses \ref{f0}, \ref{f1}, \ref{f2},
\ref{f3}, and that $V$ satisfies \ref{V0} and one of \ref{V1},
\ref{V2} or \ref{V3}. Then there exists $\e_0>0$ such that
\eqref{eq:Sch} admits a positive solution $u_{\e}$ for $\e \in (0,
\e_0)$. Moreover, there exists $\{y_{\e}\} \subset \R^N$ such that
$\e y_{\e} \to 0$ and:
$$   u_{\e}(\e ( \cdot + y_{\e}))  \to U \mbox{ in } H^1(\R^N),  $$
where $U$ is a ground state solution for
\begin{equation*}% \label{eq:limit2}
 - \Delta U + U = f(U).
\end{equation*}

\end{theorem}

This result can be compared with \cite{dpf3, jb-arma} as follows.
In \cite{dpf3} more general critical points of the potential
$V(x)$ are considered, but condition \ref{f4} is assumed. On the
other hand, the hypotheses on $f$ of \cite{jb-arma} are less
restrictive than ours, but \cite{jb-arma} considers only local
minima of $V$.

The rest of the paper is organized as follows. In Section
\ref{Preliminaries} we will give some preliminary results, most of
them well-known, related to some autonomous limit problems. We
will also define the truncation of the problem that will be used
throughout the paper. The min-max argument is exposed in Section
\ref{The min-max argument}. There we will prove the main estimate
needed for our argument, stated in Proposition \ref{fund}. This
estimate will imply the existence of a solution for the truncated
problem. In Section \ref{Asymptotic behavior} some asymptotic
estimates on the solutions will be given: in particular we will
show that the solutions of the truncated problem actually solve
our original problem. Finally, in Section \ref{appendix} some
possible extensions of our result will be briefly commented, and
some technicalities are explained in detail.

\subsection*{Acknowledgement} This work has been partially carried out during a stay of Pietro d'Avenia and Alessio Pomponio in Granada.
They would like to express their deep gratitude to the
Departamento de An{\'a}lisis Matem{\'a}tico for the support and warm
hospitality.

\section{Preliminaries} \label{Preliminaries}
In this section we will give some preliminary definitions and
results that will be used in our arguments. First, we will define
a certain truncation of $f(u)$, and establish the basic properties
of the related problem. After that, we will address the study of
certain limit problems that will appear naturally in later proofs.

Let us first fix some notation. In $\R^N$, $B(x,R)$ will denote
the usual euclidean ball centered at $x \in \R^N$ and with radius
$R>0$. Given any set $\Lambda \subset \R^N$, its complement is
denoted by $\Lambda^c$. Moreover, for any $\e>0$, we write:
$$ \Lambda^\eps := \e^{-1} \Lambda = \left\{ x\in\RN\;\vline\; \eps x \in
\Lambda \right\}. $$

In what follows we will denote by $\| \cdot \|$ the usual norm of
$\H$: other norms, like Lebesgue norms, will be indicated with a
subscript. If nothing is specified, strong and weak convergence of
sequences of functions are assumed in the space $\H$.

In our estimates, we will frequently denote by $C>0$, $c>0$ fixed
constants, that may change from line to line, but are always
independent of the variable under consideration. We also use the
notations $O(1), o(1), O(\e), o(\e)$ to describe the asymptotic
behaviors of quantities in a standard way.

\subsection{The truncated problem} By making the change of variable $x \mapsto \e x$,
problem \eqref{eq:Sch} becomes:
\begin{equation}
\label{eq:11} - \Delta u + V(\e x) u = f(u) \qquad \hbox{in } \RN.
\end{equation}

In what follows it will be useful to extend $f(u)$ as $0$ for
negative values of $u$. Observe that, by the maximum principle,
any nontrivial solution of \eqref{eq:11} will be positive, so that
we come back to our original problem.

It is well-known that solutions of \eqref{eq:11} correspond to
critical points of the functional $I_{\e}:\H \to \R$,
\[
I_\eps (u) = \frac{1}{2}\irn |\n u|^2 + \frac{1}{2} \irn V(\eps x)
u^2 - \irn F(u).
\]
However, we will not deal with \eqref{eq:11} and $I_{\e}$
directly. First, we will use a convenient truncation of the
nonlinear term $f(u)$, in the line of \cite{dpf1, dpf2, dpf3,
jt-cv}. The idea is to localize the problem around $0$, so that
the energy functional becomes coercive far from the origin. By
using min-max arguments we will find a solution of the truncated
problem. In Section \ref{Asymptotic behavior} we will show that
such solution actually solves \eqref{eq:11}.

Let us define:
\[
\tilde{f}(s)= \left\{
\begin{array}{lll}
\min\{f(s),as\} & & s\ge 0\\
0 & &s<0
\end{array}
\right.
\]
with
\begin{equation} \label{eq:a}
0< a<\left(1-\frac{2}{\mu}\right)\a_1.
\end{equation}

We also define the primitive $\tilde{F}(s)=\int_0^s \tilde{f}(t)
dt$.

In the following we will consider the balls
$B_i:=B(0,R_i)\subset\RN$ ($i=0,\ldots,4$) with $R_i <R_{i+1}$ for
$i=0,1,2,3$, where $R_i$ are small positive constants to be
determined. For technical reasons, we will choose $R_1$ such that:
\begin{equation} \label{elenuzza} \forall x \in \partial
B_1 \mbox{ with } V(x)=1, \partial_{\tau}V(x) \neq 0, \mbox{ where
} \tau \mbox{ is tangent to } \partial B_1 \mbox{ at }x.
\end{equation}

In cases \ref{V1} and \ref{V2} it is clear that such a choice is
possible. In case \ref{V3} this is also true, see Proposition
\ref{sard} in the Appendix. Observe that this is the unique point
where the $C^{N-1}$ regularity of $V$ is needed in case \ref{V3}.

Next we define $\chi : \R^N \to \R$,
\begin{equation} \label{def:chi}
\chi(x)= \left \{
\begin{array}{ll}
 1 & x \in B_1, \\
 \frac{R_2 - |x|}{R_2-R_1} & x \in B_2 \setminus B_1, \\
 0 & x \in B_2^c,
\end{array}
\right.
\end{equation}
and then
\begin{align*}
g(x,s)=&\chi(x) f(s) + \left( 1- \chi(x)\right) \tilde{f}(s),\\
G (x,s):=&\int_0^s g(x,t) dt = \chi(x) F(s) + \left( 1-
\chi(x)\right) \tilde{F}(s).
\end{align*}

We denote with subscripts the dilation of the previous functions;
being specific,
\[\chi_\eps(x)=\chi(\eps x), \] and
\[
g_\eps (x,s):= g(\e x, s)= \chi_\eps(x) f(s) + \left( 1-
\chi_\eps(x)\right) \tilde{f}(s).
\]

So, in this section we consider the truncated problem:
\begin{equation}
\label{eq:truncated} - \Delta u + V(\e x) u = g_{\e}(x,u) \qquad
\hbox{in } \RN.
\end{equation}

As mentioned above, we will find solutions of \eqref{eq:truncated}
as critical points of the associated energy functional
$\tilde{I}_{\e}:\H \to \R$, which is defined as:
\[
\tilde{I}_\eps (u) = \frac{1}{2}\irn |\n u|^2 + \frac{1}{2} \irn
V(\eps x) u^2 - \irn G_\eps(x,u),
\]
with
\[
G_\eps (x,s):=\int_0^s g_\eps (x,t) dt = \chi_\eps(x) F(s) +
\left( 1- \chi_\eps(x)\right) \tilde{F}(s).
\]

In the next lemma we collect some properties of the functions
defined above that will be of use in our reasonings:

\begin{lemma} \label{properties}

There holds:
\begin{enumerate}[label=($\tilde{{\rm f}}$\arabic*), ref=($\tilde{{\rm f}}$\arabic*)]

\item \label{en:ft2} $\tilde{F}(s) \le \min \{ \frac{1}{2} a s^2,\
F(s)\}$; \item \label{en:ft4} there exists $r>0$ such that
$\tilde{f}(s)=f(s)$ for $s \in(0,r)$;
\end{enumerate}

\begin{enumerate}
[label=(g\arabic*), ref=(g\arabic*)] \item \label{en:g1}
$G_\eps(x,s)\le F(s)$ for all $(x,s)\in
\mathbb{R}^N\times\mathbb{R}$; \item \label{en:g2}
$g_\eps(x,s)=f(s)\ $ if $|s|<r$ or $\ x \in \Lambda_1^{\e}$;
\end{enumerate}

\begin{enumerate}[label=(fg), ref=(fg)]

\item \label{en:fg} for any $\delta>0$ there exists $\ C_\d
>0$ such that:
$$ |f(s) |\le \d |s| + C_\d |s|^p,  $$
and the same assertion also holds for $\tilde{f}(s)$,
$g_{\e}(x,s)$.
\end{enumerate}

\end{lemma}

\begin{proof}
Properties \ref{en:ft2}, \ref{en:ft4}, \ref{en:g1}, \ref{en:g2}
follow immediately from the definitions of $\tilde{f}$ and
$g_{\e}$. Finally, property \ref{en:fg} follows from the
assumptions \ref{f1} and \ref{f2} made on $f$.

\end{proof}

%and \[ g(y,s)=\chi(y) f(s) + \left( 1- \chi(y)\right) \tilde{f}(s)
%\]
%and
%\[
%G(x,s)=\int_0^s g(y,\tau) d\tau = \chi(y) F(s) + \left( 1-
%\chi(y)\right) \tilde{F}(s).
%\]

\begin{proposition}
\label{pr:PS} For every $\eps>0$, the functional $\tilde{I}_\eps$
satisfies the Palais-Smale condition.
\end{proposition}

The proof of this result is basically identical to the proof of
\cite[Lemma 1.1]{dpf1}. We reproduce it here for the sake of
completeness.
\begin{proof}
Let $\{u_n\}$ be a (PS) sequence for $\tilde{I}_\eps$, i.e.
\[
\tilde{I}_\eps (u_n)=\frac{1}{2}\irn |\n u_n|^2 + \frac{1}{2} \irn
V(\eps x) u_n^2 - \irn G_\eps(x,u_n) \to c
\]
and
\[
\tilde{I}'_\eps (u_n)[u_n]=\irn |\n u_n|^2 + \irn V(\eps x) u_n^2
- \irn g_\eps(x,u_n) u_n = o(\|u_n\|).
\]
Then, by \eqref{eq:a} we have
\begin{align*}
\mu \tilde{I}_\eps(u_n) - \tilde{I}'_\eps(u_n)[u_n] = &
\left( \frac{\mu}{2}-1\right) \irn \left ( |\n u_n|^2  + V(\eps x) u_n^2 \right ) - \irn \chi_\eps (x) \left(\mu F(u_n) - f(u_n) u_n\right)\\
& - \irn \left(1- \chi_\eps (x)\right) \left(\mu \tilde{F}(u_n) - \tilde{f}(u_n)u_n\right)\\
\ge & \left( \frac{\mu}{2}-1\right) \irn \left ( |\n u_n|^2  + V(\eps x) u_n^2 \right ) - \frac{\mu}{2}a  \irn u_n^2\\
\ge & c \| u_n\|^2.
\end{align*}
Then $\{u_n\}$ is bounded and hence $u_n \rightharpoonup u$ up to
a subsequence. Now we show that this convergence is strong. It is
sufficient to prove that for every $\d >0$ there exists $R>0$ such
that
\[
\limsup_n \|u_n\|_{H^1(B(0,R)^c)}<\d.
%\int_{\RN\setminus B_{R}} |\n u_n|^2  + V(\eps x) u_n^2 <\d.
\]
We take $R>0$ such that $B_2^\eps \subset B(0,R/2)$. Let $\phi_R$
a cut-off function such that $\phi_R=0$ in $B(0,R/2)$, $\phi_R=1$
in $B(0,R)^c$, $0\le\phi_R\le 1$ and $|\n\phi_R|\le C/R$. Then
\[
\tilde{I}'_\eps (u_n)[\phi_R u_n]=\irn \n u_n \cdot \n (\phi_R
u_n) + \irn V(\eps x) u_n^2 \phi_R  - \irn g_\eps(x,u_n) \phi_R
u_n = o_n(1)
\]
since $\{u_n\}$ is bounded. Therefore
\begin{align*}
\irn \left(|\n u_n|^2 + V(\eps x) u_n^2\right) \phi_R
= &  \irn \tilde{f}(u_n) \phi_R u_n - \irn u_n \n u_n \cdot \n \phi_R + o_n(1)\\
\le & \ a \irn  u_n^2 +\frac CR + o_n(1)
\end{align*}
and so
$$
\| u_n \|^2_{H^1(B(0,R)^c)} \le C/R +o_n(1).$$

\end{proof}

\subsection{The limit problems}

Let us start by studying the limit problem:
\begin{equation}
\label{eq:lp} \tag{$\mathcal{LP}_k$} -\Delta u + k u = f(u)
\end{equation}
for some $k >0$. The associated energy functional $\Phi_k : \H \to
\R$ is defined as:
\begin{equation} \label{def:phi}
\Phi_k(u)=\frac{1}{2} \irn |\n u|^2 + \frac{k}{2} \irn u^2 - \irn
F(u).
\end{equation}
Problem \eqref{eq:lp} can be attacked by using the Mountain Pass
Theorem in a radially symmetric framework, see \cite{ar, BGK,
beres-lions}. Indeed, let us define:
\begin{equation} \label{def:m}
m_k=\inf_{\g\in\G_k}\max_{t\in[0,1]} \Phi_k(\g(t)),
\end{equation}
with $\G_k=\{ \g\in C([0,1],\H): \ \g(0)=0, \Phi_k(\g(1))<0 \}$.
It can be proved that $m_k$ is a critical value of $\Phi_k$, that
is, there exists a solution $U \in \H$ of \eqref{eq:lp} such that
$\Phi_k(U)=m_k$. Moreover, it is known that $U$ is a ground state
solution or, in other words, it is the solution with minimal
energy, see \cite{jt-pams}.

However, without some additional hypotheses on $f$, it is not
known whether that solution is unique or not. Every non-negative
solution $U$ of \eqref{eq:lp} satisfies the following properties
(see \cite{beres-lions, strauss}):

\begin{itemize}

\item $U(x)>0$, $U$ is $C^{\infty}$ and radially symmetric;

\item $U(r)$ is decreasing in $r=|x|$ and converges to zero
exponentially as $r \to + \infty$.

\item The Pohozaev identity holds:
\begin{equation*}%\label{eq:poholim}
\frac{N-2}{2}\irn |\n U|^2 + k \frac{N}{2}\irn U^2 = N \irn F(U).
\end{equation*}

\end{itemize}

In \cite{jt-pams} it is also proved that the infimum in
\eqref{def:m} is actually a minimum. Being more specific, we have:
\begin{lemma}  {\bf (\cite[Lemma 2.1]{jt-pams})} \label{l:gamma}
Let $U\in \H$ a ground state solution of \eqref{eq:lp}. Then,
there exists $\g\in\G_k$ such that $U\in\g([0,1])$ and
\[
\max_{t\in[0,1]} \Phi_k(\g(t))=m_k.
\]

\end{lemma}

Let us briefly describe the construction in \cite{jt-pams}. Given
$t>0$, we denote:
\[
U_t=U\left(\frac{\cdot}{t}\right), \qquad t>0.
\]

For $N\ge 3$, the curve $\g$ is constructed by simply dilating the
space variable; indeed, for $\theta$ large enough,
\[
{\g}(t)= \left\{
\begin{array}{lll}
U_t & & \hbox{if }t \in (0, \theta],\\
0   & & \hbox{if }t=0.
\end{array}
\right.
\]

For $N=2$ the construction combines dilation and multiplication by
constants in a certain way:
\[\left \{
\begin{array}{lll}
tU_{\t_0} & & \hbox{if }t\in[0,1],\\
U_\t       & & \hbox{if } \t\in[\t_0,\t_1],\\
tU_{\t_1} & & \hbox{if }t\in[1,\t_2].
\end{array}
\right.
\]
with suitable $\t_0\in(0,1)$ and $\t_1,\t_2>1$. Observe that in
both cases $\gamma$ is defined in a closed interval: a suitable
re-parametrization of it gives us the desired curve.

This curve will be of use for the construction of our min-max
scheme.

\bigskip

The following lemma studies the dependence of the critical level
on $k$:

\begin{lemma}
\label{le:cm} The map $m:(0,+\infty) \to (0,+\infty)$,
\[
m(k) = m_k.
\]
is strictly increasing and continuous.
\end{lemma}
\begin{proof}
Let us first show that $m$ is strictly increasing. Take $k_1, k_2
>0$ with $k_1 < k_2$ and $ \g \in \G_{k_2}$ given by Lemma
\ref{l:gamma}. Observe that clearly $\g \in \G_{k_1}$, and:
$$ m_{k_1} \le \max_{t\in[0,1]} \Phi_{k_1}(\g(t)) < \max_{t\in[0,1]}
\Phi_{k_2}(\g(t))= m_{k_2}.$$

We now prove the continuity of $m$. Take $\{k_j\}_j$ a sequence of
positive real numbers that converges to $k>0$. As above, take
$\gamma \in \G_k$ given by Lemma \ref{l:gamma}: then, for $j$
large enough $\gamma \in \G_{k_j}$ and
$$ m_{k_j} \le \max_{t\in[0,1]} \Phi_{k_j}(\g(t)) \to \max_{t\in[0,1]}
\Phi_{k}(\g(t))= m_{k}.$$

Then
\begin{equation*}%\label{eq:usc}
\limsup_j m_{k_j} \le m_k.
\end{equation*}

We now prove a reversed inequality. For every $j\in\mathbb{N}$, we
consider $U_j$ a radially symmetric least energy solution of
\begin{equation}
\label{eq:lpkj} \tag{$\mathcal{LP}_{k_j}$} -\Delta u + k_j u =
f(u).
\end{equation}

The sequence $\{U_j\}_j$ is bounded in $\H$-norm. Indeed, since
\[
\frac{1}{2} \irn |\n U_j|^2 + \frac{k_j}{2} \irn U_j^2 - \irn
F(U_j)={m}_{k_j}= O(1),
\]
and
\[
\irn |\n U_j|^2 + k_j \irn U_j^2 - \irn f(U_j)U_j=0
\]
then, we obtain by \ref{f3}
\begin{align*}
\mu {m}_{k_j} =  & \left(\frac{\mu}{2}-1\right) \irn |\n U_j|^2 + k_j \left(\frac{\mu}{2}-1\right) \irn U_j^2 - \irn \mu F(U_j) - f(U_j)U_j\\
\ge  & c \| U_j\|^2.
\end{align*}
Therefore $U_j \rightharpoonup U$ in $\Hr= \{ u \in \H \; \vline \; \ u \mbox{ is radially symmetric}\}$. By the compact embedding of $\Hr$ into $L^{p+1}(\RN)$  (see \cite{strauss}), we get that $U_j \to U$ in $L^{p+1}(\RN)$.\\
Since
\[
\frac{1}{2} \irn |\n U_j|^2 + \frac{k_j}{2} \irn U_j^2 = {m}_{k_j}
+ \irn F(U_j) \ge c > 0
\]
and, by \ref{en:fg}, fixed $\delta>0$ small enough,
\[
0< c \le \irn |\n U_j|^2 + (k_j -\d) \irn U_j^2 \le C \irn
|U_j|^{p+1},
\]
so that $U\neq 0$.

Observe that $U$ is a positive radially symmetric solution of the
problem:
$$ - \Delta U + k U = f(U).$$

Moreover, using the strong convergence in $L^{p+1}(\RN)$,
\[
\irn F(U_j) \to \irn F(U).
\]

By the lower semicontinuity of the $\H$ norm, we conclude:
$$ \liminf_{j \to +\infty} m_{k_j}= \liminf_{j \to +\infty}
\Phi_{k_j}(U_j)\ge \Phi_k(U) \ge m_k.$$
\end{proof}

We finish the section with a couple of definitions that will be of
use later. First, let us restrict ourselves to the case $k=1$; for
simplicity, we will denote:
\begin{equation*} %\label{joe}
 \Phi:= \Phi_1,\ m:= m_1.
\end{equation*}
Let us define:
\begin{equation}\label{eq:S}
\mathcal{S}=\left\{ u\in\H\;\vline\; \Phi(u)=m, \
\Phi'(u)=0\right\}.
\end{equation}
In other words, $\mathcal{S}$ denotes the set of positive ground
state solutions of the problem:
$$ - \Delta u + u = f(u).$$

Moreover, given any $ y \in \R^N$, we define the energy functional
$J_y : \H \to \R$,
\begin{equation} \label{def:J}
J_y (u) = \frac{1}{2}\irn |\n u|^2 + \frac{V(y)}{2} \irn  u^2 -
\irn G(y,u).
\end{equation}
Obviously, the critical points of $J_y$ are solutions of the
problem:
$$ -\Delta u  + V(y) u = g(y,u).  $$

\section{The min-max argument} \label{The min-max argument}

In this section we will develop the min-max argument that will
provide the existence of a solution. In order to do that, some
estimates on the min-max value are needed: those are the
fundamental part of our work, and are contained in Proposition
\ref{fund}. A key ingredient of our proof is a classical property
of the Brouwer degree concerning existence of connected sets of
solutions (see the proof of Lemma \ref{le:grado}).

First of all, let us observe that under our hypotheses on $V$,
there exists a vector space $E$ such that:
\begin{enumerate}[label=\alph*), ref=\alph*)]
\item $V|_E$ has a strict local maximum at $0$; \item
$V|_{E^\perp}$ has a strict local minimum at $0$.
\end{enumerate}

Indeed, in case \ref{V1} $E= \RN$, whereas, in case \ref{V2} $E$
is the space formed by eigenvectors associated to negative
eigenvalues of $D^2V(0)$.

\medskip

First of all, let us define the following topological cone:
\begin{equation} \label{def:cono}
\Ce=\left\{ \gamma_t(\cdot-\xi)\;\vline\; t\in[0,1], \xi\in
\overline{B_0^\eps}\cap E \right\}.
\end{equation}

Here $\gamma_t=\gamma(t)$ is the curve given in Lemma
\ref{l:gamma} for $k=1$ and $U$ a radially symmetric ground state.
Observe that $\gamma(0)=0$ is the vertex of the cone. Let us
define a family of deformations of $\Ce$:
\[
\G_\eps=\left\{\eta\in C\left(\Ce,\H\right) \; \vline \;
\eta(u)=u, \ \forall  u \in \partial\Ce \cup
\tilde{I}_{\e}^{m/2}\right\},
\]
where $\partial \Ce$ is the topological boundary of $\Ce$ and
$\tilde{I}_{\e}^{m/2}$ is the sub-level $\tilde{I}_{\e}^{m/2}= \{
u \in \H\mid  \tilde{I}_{\e}(u) < m/2\}$. Recall that $m=m_1$ is
the ground state energy level of the problem $ - \Delta u + u
=f(u)$, see \eqref{def:m}.

We define the min-max level:
\[
m_\eps=\inf_{\eta\in\G_\e}\max_{u\in\Ce} \tilde{I}_\eps(\eta(u)).
\]

\begin{proposition}
\label{pr:bb} There exist $\eps_0>0, \ \delta>0$ such that for
every $\eps\in(0,\eps_0)$
%for every $\eps>0$ %and $\xi \in B_1^\eps$
\begin{equation*} %\label{boundary}
\tilde{I}_\eps|_{\partial \Ce} \le m-\delta.
\end{equation*}
\end{proposition}
\begin{proof}
It suffices to show that:
\begin{equation} \label{1}
\tilde{I}_\eps(\gamma_1(\cdot - \xi)) <0 \quad \forall  \xi \in
B_0^\e \cap E
\end{equation}
and
\begin{equation} \label{2}
\tilde{I}_\eps(\gamma_t(\cdot - \xi)) < m - \d \quad \forall  \xi
\in
\partial B_0^{\e} \cap E, \ t \in [0,1].
\end{equation}

Let us denote $\hat{U}= \gamma_1(\cdot-\xi)$ for some $\xi \in
B_0^\e \cap E$. Then,
$$
\tilde{I}_\eps(\hat{U}) \le  \int_{B_1^{\e}} \left ( \frac{1}{2}
|\n \hat{U}|^2 + \frac{1}{2} V(\eps x) \hat{U}^2 - F(\hat{U})
\right )\,  + \int_{(B_1^{\e})^c} \left ( \frac{1}{2} |\n
\hat{U}|^2 + \frac{1}{2} V(\eps x) \hat{U}^2 \right )\, .$$ By the
exponential decay of $\hat{U}$, we get:
$$\tilde{I}_\eps(\hat{U}) \le \Phi_{\nu}(\hat{U})+ o_\eps(1)$$ where $\nu= \max_{x \in \overline{B_1}}
V(x)$ and $\Phi_{\nu}$ is defined in \eqref{def:phi} . By
shrinking $B_1$, if necessary, we can assume that
$\Phi_{\nu}(\hat{U})$ is negative, so we obtain \eqref{1}.

In order to prove \eqref{2}, let us first observe that there
exists $\sigma>0$ such that $V(x) < 1-\sigma$ for every $x \in
\partial B_0^\e \cap E$. Then,
\begin{align*}
\tilde{I}_\eps(\gamma_t(\cdot-\xi)) &\le  \int_{B(0, \sqrt{\e})}
\left ( \frac{1}{2} |\n \gamma_t(x)|^2 + \frac{1}{2} V(\eps
(x+\xi))\g_t(x)^2 - F(\gamma_t(x)) \right )
\\
&\quad  + \int_{B(0, \sqrt{\e})^c} \left ( \frac{1}{2} |\n
\gamma_t(x)|^2 + \frac{1}{2} \alpha_2\g_t(x)^2 \right ).
\end{align*}

Again by the exponential decay of $\gamma_t$, the second right
term tends to zero as $\e \to 0$. Observe also that this
convergence is uniform in $t$, since the exponential decay is
uniform in $t$. By using dominated convergence theorem,
$$
\tilde{I}_\eps(\gamma_t(\cdot-\xi)) \le \Phi_{1-\sigma}(\gamma_t)
+ o_\e(1).$$ Finally, since $\Phi_{1-\sigma}(u) < \Phi(u)$ for any
$u \neq 0$, we have that
$$ \max_{t \in [0,1]} \Phi_{1-\sigma}(\gamma_t) < \max_{t \in [0,1]} \Phi(\gamma_t)=m. $$

\end{proof}

We now give a first estimate on the min-max values:

\begin{proposition}\label{estima}
We have that
\[
\limsup_{\eps\to 0} m_{\e} \le m.
\]
\end{proposition}

\begin{proof}
By definition,
$$ m_{\e} \le \max_{u \in \Ce}\tilde{I}_\e(u).$$
So, let us estimate this last term. In the following we take a
sequence $\e=\e_n \to 0$, but we drop the sub-index $n$ for the
sake of clarity. For any $\e>0$ sufficiently small, there exists
$t_{\e} \in [0,1]$, $\xi_\e \in \overline{B_0^\e} \cap E$ such
that:
\begin{align*}
\max_{u \in \Ce}\tilde{I}_\e(u) &=
\tilde{I}_\e(\gamma_{t_\e}(\cdot - \xi_\e))
\\
&\le  \int_{B(0, \sqrt{\e})} \left ( \frac{1}{2} |\n
\gamma_{t_\e}(x)|^2 + \frac{1}{2} V(\eps
(x+\xi_\e))\gamma_{t_\e}(x)^2 - F(\gamma_{t_\e}(x)) \right )\,
\\
&\quad + \int_{B(0, \sqrt{\e})^c} \left ( \frac{1}{2} |\n
\gamma_{t_\e}(x)|^2 + \frac{1}{2} \alpha_2 \gamma_{t_\e}^2
\right)\, .
\end{align*}
Up to a subsequence we can assume that $t_\e \to t_0 \in [0,1]$
and $\e \xi_\e \to x_0 \in \overline{B_0} \cap E$. Therefore, by
the uniform exponential decay of $\gamma_t$ and dominated
convergence theorem, we get:
$$\tilde{I}_\e(\gamma_{t_\e}(\cdot - \xi_\e)) \to \Phi_{V(x_0)}(\gamma_{t_0}). $$
\\
Observe now that $V(x_0) \le 1$ and then
$$ \Phi_{V(x_0)}(\gamma_{t_0}) \le \max_{t \in [0,1]}
\Phi(\gamma_t) = m.$$

\end{proof}

The following proposition yields a fundamental estimate in our
min-max argument:

\begin{proposition}\label{fund}
There holds
\[
\liminf_{\eps\to 0} m_{\e} \ge m.
\]
\end{proposition}

Before proving Proposition \ref{fund}, let us show how it is used
to provide existence of a solution. The following theorem is the
main result of this section:

\begin{theorem} There exists $\e_0>0$ such that for $\e \in
(0,\e_0)$ there exists a positive solution $u_{\e}$ of the problem
\eqref{eq:truncated}. Moreover, $\tilde{I}_{\e}(u_\e)=m_\e$.

\end{theorem}

\begin{proof}
 By Propositions \ref{estima} and \ref{fund}, we deduce that $m_{\e} \to m$ as $\e \to 0$. From Proposition \ref{pr:bb} we get that for small
 values of $\e$, $m_\e > \max_{\partial \Ce} \tilde{I}_{\e}$.
 Moreover, recall that $\tilde{I}_\e$ satisfies the (PS)
 condition, see Proposition \ref{pr:PS}. Therefore, classical min-max theory implies that $m_\e$ is a critical value
 of $\tilde{I}_\e$; let us denote $u_\e$ a critical point.
 Finally the fact that $u_\e$ is positive follows from
 the maximum principle.

\end{proof}

\subsection{Proof of Proposition \ref{fund}}

\

The rest of the section is devoted to the proof of Proposition
\ref{fund}. This proof will be divided in several lemmas and
propositions.

First we define $\pi_E$ as the orthogonal projection on $E$ and we
set $h_\eps : \RN \to E$ defined as $h_\eps (x)=\pi_E(x)
\chi_{B_3^\eps}(x)$, where $\chi_{B_3^\eps}$ is the characteristic
function related to $B_3^\eps$. Let us define a barycentre type
map $\b_\eps: \H\setminus \{0\} \to E$ such that for any $u\in
\H\setminus \{0\}$
\[
\b_\eps (u) =\frac{\irn h_\eps(x) u^2 \,d x}{\irn u^2 \, d x }.
\]

For a fixed $\d>0$ sufficiently small, let us define
\[
\Xi_\eps=\left\{ \Sigma \subset \H\setminus\{0\} \left|
\begin{array}{l}
\Sigma \hbox{ is connected and compact}
\\
\exists u_0,u_1 \in \Sigma \hbox{ s.t. }  \|u_0\|\le \d , \tilde
I_\eps(u_1)<0
\\
\forall u \in \Sigma,\  \b_\eps(u)=0
\end{array}
\right. \right\}.
\]

Let us observe that we have to require that $0 \notin \Sigma$
because the barycentre $\b_\eps$ is not well defined in $0$. We
also define the corresponding min-max value:
$$b_\eps=\inf_{\Sigma \in \Xi_\eps} \max_{u\in \Sigma}\tilde I_\eps
(u).$$

Observe that, since $\tilde{I}_\e \ge \Phi_{\alpha_1}$, we have:
$$ b_\e \ge m_{\alpha_1}>0.$$

\begin{lemma}\label{le:grado}
There exists $\e_0>0$ such that for any $\eps \in (0 ,\e_0)$ and
for any $\eta \in \G_\eps$ there exists $\Sigma\in \Xi_\eps$ such
that $\Sigma \subset \eta(\Ce)$.
\end{lemma}

\begin{proof}
Let us take $t_0>0$ sufficiently small, and $\eta \in \G_\eps$.
For any $t\in [t_0,1]$, we define $\psi^\e_t: \overline{B_0^\eps}
\cap E\to E$ such that
\[
\psi^\e_t(\xi)=\b_\eps\Big(\eta\big(\g_t(\cdot - \xi)\big)\Big).
\]
Let us observe that, by the properties of $\eta \in \G_\eps$,
$\eta\big(\g_t(\cdot - \xi)\big)\neq 0$, for all $t\in [t_0,1]$
and for all $\xi \in \overline{B_0^\eps} \cap E$, and so
$\psi^\e_t$ is well defined. Moreover, $\|\g_{t_0}\|$ can be made
arbitrary small by taking smaller $t_0$.
\\

Moreover, by the exponential decay of $\g_t$,
$$ \psi^\e_t(\xi) \to \xi \mbox{ uniformly in $\de B_0^\eps \cap E$ and $t\in [t_0,1]$, as } \e \to 0.$$

Therefore we can choose $\e$ small enough so that
\[
\deg(\psi^\e_t,B_0^\eps\cap E,0)=\deg(\Id,B_0^\eps\cap E,0)=1, \
\hbox{ for all }t\in [t_0,1].
\]

We can conclude that for every $t\in [t_0,1]$, there exists
$\xi\in B_0^\eps\cap E$ such that $\psi^\e_t(\xi)=0$. Moreover
there exists a connected and compact set $\Upsilon \subset [t_0,
1] \times (B_0^\e \cap E)$ that takes all values in $[t_0,1]$ and
such that $\psi^\e_t(\xi)=0$ for all $(t,\xi) \in \Upsilon$, (see
\cite{leray-s, mawhin}).

Then, it suffices to define
$$\Sigma = \{ \eta(\g_t(\cdot - \xi)) \mid (t,\xi) \in \Upsilon \}.$$

\end{proof}

As a consequence of the previous lemma, we obtain the following
inequality:
\begin{equation*} %\label{le:m>=b}
m_\eps \ge b_\eps.
\end{equation*}

Hence, the proof of Proposition \ref{fund} is completed if we
prove the following result:

\begin{proposition}\label{pr:b>=m}
We have that
\[
 \liminf_{\eps\to 0} b_\eps \ge m.
\]
\end{proposition}
In order to prove this proposition, we will need some midway
lemmas.

\begin{lemma}\label{le:lambda}
There exists $\eps_0>0$ such that, for any $\eps\in (0,\eps_0)$,
there exist $u_\eps\in \H$, with $\b_\eps(u_\eps)=0$, and
$\l_\eps\in E$ such that
\begin{equation}\label{eq:lambda}
-\Delta u_\eps+V(\eps x) u_\eps=g_\eps( x, u_\eps) +\l_\eps \cdot
h_\eps(x) u_\eps,
\end{equation}
and
\[
\tilde I_\eps (u_\eps)=b_\eps.
\]
Moreover, the sequence $\{u_\eps\}$ is bounded in $\H$.
\end{lemma}
\begin{proof}
Let $\eps>0$ be fixed. By classical min-max theory, there exists a
sequence $\{u_n\} \subset \H$ which is a constrained (PS) sequence
at level $b_\eps$, namely, there exists $\{\l_n\}\subset E$ such
that
\begin{align}
\tilde I_\eps(u_n) \to b_\eps, &\ \hbox{ as }n\to +\infty,
\label{eq:psl1}
\\
\tilde I_\eps'(u_n)-  \frac{\l_n \cdot h_\eps(x)u_n}{\irn u_n^2}
\to 0, &\ \hbox{ as }n\to +\infty. \label{eq:psl2}
\end{align}
Since $\b_\eps(u_n)=0$, by \eqref{eq:psl1} and \eqref{eq:psl2}
repeating the arguments of Proposition \ref{pr:PS}, we get that
$\{u_n\}$ is bounded in the $H^1-$norm, (uniformly with respect to
$\eps$) and, therefore, up to a subsequence, it converges weakly
to some $u\in \H$. This convergence is actually strong arguing as
in the proof of Proposition \ref{pr:PS} and choosing $R$ big
enough such that $\phi_R h_\eps=0$.

\end{proof}

\begin{lemma}\label{le:L2uL2}
There holds $u_\eps \chi_{B_2^\eps}\nrightarrow 0$ in $L^2(\RN)$
as $\eps \to 0$.
\end{lemma}

\begin{proof}
Since $u_\eps$ is a solution of (\ref{eq:lambda}) with
$\b_\eps(u_\eps)=0$, multiplying (\ref{eq:lambda}) by $u_\eps$,
integrating and using \ref{en:fg}, for a fixed sufficiently small
$\d>0$, we have
\begin{equation*}%\label{eq:ul2}
\irn |\n u_\eps|^2+V(\eps x)u_\eps^2 =\irn g_\eps( x,u_\eps)u_\eps
\le \irn (a+\d) u_\eps^2 +C\int_{B_2^\eps}u_\eps^{p+1}.
\end{equation*}
Then
%\marginpar{propriet\`a su $a$}
%\begin{align}
%\int_{B_2^\eps}g_\eps( x, u_\eps)u_\eps
%&\le  \int_{B_2^\eps} a u_\eps^2+ C|u_\eps|^{p+1},    \label{eq:gl2}
%\\
%\int_{(B_2^\eps)^c}g_\eps(x, u_\eps)u_\eps
%&\le  \int_{(B_2^\eps)^c} a u_\eps^2.   \label{eq:gl2c}
%\end{align}
%Then by (\ref{eq:ul2}), (\ref{eq:gl2}) and (\ref{eq:gl2c}), we have
\begin{equation*}
\|u_\eps\|_{L^{p+1}(B_2^\eps)}^2 \le C \|u_\eps\|^2 \le
C\|u_\eps\|_{L^{p+1}(B_2^\eps)}^{p+1},
\end{equation*}
and
\[
u_\eps \chi_{B_2^\eps}\nrightarrow 0,\quad \hbox{in }L^{p+1}(\RN).
\]
Now, by the boundedness of $\{u_\eps\}$ in $\H$ and so in
$L^s(\RN)$, for a certain $s>p+1$, we can conclude by
interpolation, indeed, for a suitable $\a<1$:
\[
0<c\le \|u_\eps\|_{L^{p+1}(B_2^\eps)} \le
\|u_\eps\|_{L^{2}(B_2^\eps)}^{\a}
\|u_\eps\|_{L^{s}(B_2^\eps)}^{1-\a} \le C
\|u_\eps\|_{L^{2}(B_2^\eps)}^{\a}.
\]
\end{proof}

\begin{lemma}\label{le:L4c}
We have that $\|u_\eps\|_{H^1((B_4^\eps)^c)}\to 0$.
\end{lemma}

\begin{proof}
Let $\phi_\eps :\RN\to \R$ be a smooth function such that
\[
\phi_\eps(x)= \left\{
\begin{array}{ll}
0 & \hbox{in }B_3^\eps,
\\
1 & \hbox{in }(B_4^\eps)^c,
\end{array}
\right.
\]
and with $0\le  \phi_\eps \le 1$ and $|\n \phi_\eps|\le C\eps$.
\\
By Lemma \ref{le:lambda}, since $\phi_\eps h_\eps=0$, we have that
\[
\tilde I_\eps'(u_\eps)[\phi_\eps u_\eps]=0,
\]
namely, by definition of $g_\e$,
\[
\irn (|\n u_\eps|^2+V(\eps x)u_\eps^2)\phi_\eps +\irn u_\eps \n
u_\eps \cdot \n \phi_\eps =\irn g_\eps( x,u_\eps)u_\eps \phi_\eps
\le  \irn a u_\eps^2 \phi_\e,
\]
and so we can conclude observing that
\[
\int_{(B_4^\eps)^c}|\n u_\eps|^2+u_\eps^2 \le C\eps.
\]
\end{proof}

\begin{lemma}\label{le:Oeps}
We have that $\l_\eps=O(\eps)$.
\end{lemma}

\begin{proof}
In the sequel we can suppose that $\l_\eps\neq 0$, otherwise the
lemma is proved. Let us denote $\tilde \l_\eps=\l_\eps /|\l_\eps
|$.
\\
Let $\phi_\eps:\RN\to \R$ be a smooth function such that
\[
\phi_\eps(x)=\left\{
\begin{array}{ll}
1 & \hbox{in }B_2^\eps,
\\
0 & \hbox{in }(B_3^\eps)^c,
\end{array}
\right.
\]
with $0\le \phi_\eps \le 1$ and $|\n \phi_\eps|\le C\eps$.
\\
We follow an idea of \cite{esteban,esteban2}. By regularity
arguments $u_\eps \in H^2(\RN)$ and then we are allowed to
multiply \eqref{eq:lambda} by $\phi_\eps \de_{\tilde \l_\eps}
u_\eps$ and to integrate by parts. Then
\begin{align}
&\int_{B_3^\eps}\big[\n u_\eps\cdot \n(\de_{\tilde \l_\eps} u_\eps
)\big]\phi_\eps +\int_{B_3^\eps\setminus B_2^\eps}(\n u_\eps\cdot
\n \phi_\eps )(\de_{\tilde \l_\eps} u_\eps ) \nonumber
\\
&\qquad +\int_{B_3^\eps}V(\eps x)u_\eps(\de_{\tilde \l_\eps}
u_\eps )\phi_\eps -\int_{B_3^\eps} g_\eps( x, u_\eps)(\de_{\tilde
\l_\eps} u_\eps )\phi_\eps \nonumber
\\
&\quad =\int_{B_3^\eps}(\l_{\eps} \cdot h_\eps(x))
u_\eps(\de_{\tilde \l_\eps} u_\eps )\phi_\eps. \label{eq:**}
\end{align}
Let us evaluate each term of the previous equality. We have
\[
0=\irn \de_{\tilde \l_\eps} \big[|\n u_\eps|^2 \phi_\eps\big]
=2\irn\big[\n u_\eps\cdot \n(\de_{\tilde \l_\eps} u_\eps
)\big]\phi_\eps +\irn |\n u_\eps|^2 \de_{\tilde \l_\eps} \phi_\eps
,
\]
and so
\begin{equation}\label{eq:lnu}
\int_{B_3^\eps}\big[\n u_\eps\cdot \n(\de_{\tilde \l_\eps} u_\eps
)\big]\phi_\eps =-\frac{1}{2}\int_{B_3^\eps\setminus B_2^\eps}|\n
u_\eps|^2 \de_{\tilde \l_\eps} \phi_\eps =O(\eps).
\end{equation}
Easily we have
\begin{equation}\label{eq:lnchi}
\int_{B_3^\eps\setminus B_2^\eps}(\n u_\eps\cdot \n \phi_\eps
)(\de_{\tilde \l_\eps} u_\eps )=O(\eps).
\end{equation}
Analogously, we have
\begin{align*}
0&=\irn \de_{\tilde \l_\eps} \big[V(\eps x)u_\eps^2 \phi_\eps\big]
\\
&=\eps \irn (\de_{\tilde \l_\eps} V(\eps x))u_\eps^2 \phi_\eps +2
\irn  V(\eps x)(\de_{\tilde \l_\eps} u_\eps )u_\eps\phi_\eps +\irn
V(\eps x) u_\eps^2 (\de_{\tilde \l_\eps} \phi_\eps ),
\end{align*}
and so
\begin{equation}\label{eq:lv}
\int_{B_3^\eps}  V(\eps x)u_\eps(\de_{\tilde \l_\eps} u_\eps
)\phi_\eps =-\frac{\eps}2 \int_{B_3^\eps} (\de_{\tilde \l_\eps}
V(\eps x))u_\eps^2 \phi_\eps -\frac 12\int_{B_3^\eps\setminus
B_2^\eps} V(\eps x) u_\eps^2 (\de_{\tilde \l_\eps} \phi_\eps )
=O(\eps).
\end{equation}
Moreover, since by the definition of $G_\eps$,
\[
\de_{\tilde \l_\eps}  G_\eps( x,u_\eps) = \eps \de_{\tilde
\l_\eps}  \chi(\eps x)(F(u_\eps)-\tilde F(u_\eps)) +g_\eps(
x,u_\eps)\de_{\tilde \l_\eps}  u_\eps,
\]
we have
\begin{align*}
0&=\irn \de_{\tilde \l_\eps} \big[G_\eps( x,u_\eps) \phi_\eps\big]
\\
&=\eps\irn (F(u_\eps)-\tilde F(u_\eps))(\de_{\tilde \l_\eps} \chi
(\eps x)) \phi_\eps +\irn g_\eps( x,u_\eps)(\de_{\tilde \l_\eps}
u_\eps ) \phi_\eps +\irn G_\eps (x,u_\eps) (\de_{\tilde \l_\eps}
\phi_\eps )
\end{align*}
and so
\begin{equation}\label{eq:lg}
\int_{B_3^\eps}  g_\eps( x,u_\eps)(\de_{\tilde \l_\eps} u_\eps )
\phi_\eps =-\eps\int_{B_3^\eps}  (F(u_\eps)-\tilde
F(u_\eps))(\de_{\tilde \l_\eps} \chi (\eps x)) \phi_\eps
-\int_{B_3^\eps\setminus B_2^\eps}  G_\eps (x,u_\eps) (\de_{\tilde
\l_\eps} \phi_\eps ) =O(\eps).
\end{equation}
Finally
\begin{align*}
0&=\int_{B_3^\eps}\de_{\tilde \l_\eps} \big[(\l_{\eps} \cdot
h_\eps(x)) u^2_\eps \phi_\eps\big]
\\
&= |\l_\eps|\int_{B_3^\eps}u^2_\eps \phi_\eps +2\int_{B_3^\eps}
(\l_{\eps} \cdot h_\eps(x)) u_\eps(\de_{\tilde \l_\eps} u_\eps
)\phi_\eps +\int_{B_3^\eps\setminus B_2^\eps} (\l_{\eps} \cdot
h_\eps(x)) u_\eps^2(\de_{\tilde \l_\eps} \phi_\eps )
\end{align*}
and so
\begin{equation}\label{eq:lh}
\int_{B_3^\eps}  (\l_{\eps} \cdot h_\eps(x)) u_\eps(\de_{\tilde
\l_\eps} u_\eps )\phi_\eps =-\frac{1}2|\l_{\eps}|\int_{B_3^\eps}
u^2_\eps \phi_\eps +O(\eps).
\end{equation}
By (\ref{eq:**})--(\ref{eq:lh}) and by Lemma \ref{le:L2uL2}, we
conclude.
\end{proof}

Therefore, we can suppose that there exists $\bar \l\in E$ such
that
\[
\bar \l=\lim_{\eps\to 0} \frac{\l_\eps}{\eps}.
\]

\begin{proof}[Proof of Proposition \ref{pr:b>=m}]
We will consider separately the case $\bar \l=0$ and $\bar \l\neq
0$.
\\
\
\\
{\bf Case 1)}: $\bar \l=0$.

We consider a sequence $\e_k\to 0$, that we still denote by $\e$.
\\
By \cite[Proposition 4.2]{jt-cv}, there exists $n\in \N$, $\bar
c>0$ and, for all $i=1,\ldots,n$,  there exist $y_\eps^i\in
B_2^\eps $, $\bar y_i \in B_2$ and $u_i\in \H \setminus \{0\}$
such that
\begin{align*}
&\eps y_\eps^i\to \bar y_i, %\nonumber
\\
&|y_\eps^i -y_\eps^j|\to \infty, \ \hbox{ if }i\neq j ,%\label{eq:yy}
\\
&u_\eps (\cdot +y_\eps^i )\rightharpoonup u_i,\ \hbox{ weakly in }\H,  %\label{eq:uuiw}
\\
&\|u_i\|\ge \bar c,  %\label{eq:barc}
\\
& u_\eps -\sum_{i=1}^n u_i (\cdot -y_\eps^i ) \to 0 ,\ \hbox{
strongly in }\H ,
\end{align*}
and $u_i$ is a positive solution of
\[
-\Delta u_i+V(\bar y_i) u_i=g(\bar y_i,u_i).
\]
Moreover
\[
\lim_{\eps\to 0} b_\eps=\lim_{\eps\to 0} \tilde I_\eps (u_\eps) =
\sum_{i=1}^n J_{\bar y_i}(u_i).
\]
Since, by \ref{en:g1}, we have that $ J_{\bar y_i}(u_i)\ge
\Phi_{V(\bar y_i)}(u_i)$, for all $i=1,\ldots,n$, to conclude the
proof, we have only to show that
\begin{equation*}
\sum_{i=1}^n \Phi_{V(\bar y_i)}(u_i)\ge m.
\end{equation*}
This is trivially true if $n\ge 2$ by Lemma \ref{le:cm}, since
$\bar y_i\in B_2$. If, otherwise, $n=1$, since
$\b_\eps(u_\eps)=0$,
\[
0=\e \int_{B_3^\e-y^1_\e} \pi_E(x+y^1_\eps ) u_\e^2(x+y^1_\e) =
\int_{B_3^\e-y^1_\e} \pi_E(\e x+\e y^1_\e) u_\e^2(x+y^1_\e) \to
\pi_E(\bar{y}_1) \irn u_1^2.
\]
Therefore, $\bar y_1\in E^\perp$ and then Lemma \ref{le:cm}
implies:
\[
\Phi_{V(\bar y_1)}(u_1)\ge  m_{V(\bar y_1)}\ge m.
\]

\
\\
{\bf Case 2)}: $\bar \l\neq 0$.
\\
In this case we cannot conclude simply as in the previous one
because of the interference of the Lagrange multiplier. Some
technical work is needed here.

Let
\[
H_\eps=\left\{x\in \RN \mid \bar \l \cdot x \le
\frac{\a_1}{2\eps}\right\}.
\]
\begin{lemma}\label{le:L2LpH}
We have that $u_\eps \chi_{H_\eps} \nrightarrow 0$ in the
$L^2$-norm and in the $L^{p+1}$-norm.
\end{lemma}

\begin{proof}
Let $H'_\eps=\left\{x\in \RN \mid \bar \l  \cdot x\le
\frac{\a_1}{3 \eps}\right\}\subset H_\eps$.  We will prove that
\begin{equation}\label{eq:L2H1}
\int_{H'_\eps} u_\eps^2 \nrightarrow 0, \ \hbox{ as }\eps\to 0.
\end{equation}
Suppose by contradiction that
\begin{equation}\label{eq:absH}
\int_{H'_\eps} u_\eps^2 \to 0, \ \hbox{ as }\eps\to 0.
\end{equation}
Since $\b_\eps(u_\eps)=0$ and $\bar \l\in E$, we have
\[
0= \irn \bar{\l} \cdot h_\eps(x) u_\eps^2 =\int_{(H'_\eps)^c \cap
B_3^\e}\!\!\! \!\!\! \bar{\l} \cdot x \ u_\eps^2 +\int_{H'_\eps
\cap B_3^\e} \!\!\! \!\!\bar{\l} \cdot x \  u_\eps^2 \ge
\frac{\a_1}{3 \eps}\int_{(H'_\eps)^c \cap B_3^\e} \!\!\!\!\!\!
u_\eps^2 +\int_{H'_\eps \cap B_3^\e} \!\!\!\!\! \bar{\l} \cdot x \
u_\eps^2
\]
Therefore
\[
\frac{\a_1}{3 \eps} \int_{(H'_\eps)^c \cap B_3^\e}u_\eps^2 \le
\left|\int_{H'_\eps \cap B_3^\e} \bar{\l} \cdot x \
u_\eps^2\right| \le \frac{|\bar \l|R_3}{\eps}\int_{H'_\eps \cap
B_3^\e}u_\eps^2
\]
and so
\[
\int_{(H'_\eps)^c \cap B_3^\e}u_\eps^2 \to 0, \ \hbox{ as }\eps\to
0.
\]
This last formula, together with (\ref{eq:absH}), implies that
$u_\eps \chi_{B_3^\e}\to 0$ in $L^2(\RN)$ but we get a
contradiction with Lemma \ref{le:L2uL2} and so the first part of
the lemma is proved.
\\
Let us now consider the second part of the statement.
\\
Let $\phi_\eps :\RN\to \R$ be a smooth function such that
\[
\phi_\eps(x)= \left\{
\begin{array}{ll}
1 & \hbox{in }H'_\eps,
\\
0 & \hbox{in }(H_\eps)^c,
\end{array}
\right.
\]
and with $0\le\phi_\eps\le 1$ and $|\n \phi_\eps|\le C\eps$.
Multiplying (\ref{eq:lambda}) by $u_\eps \phi_\eps$ and
integrating, we have
\[
\int_{H_\eps} |\n u_\eps|^2 \phi_\eps +\int_{H_\eps\setminus
H'_\eps} \n u_\eps \cdot \n \phi_\eps u_\eps +\int_{H_\eps} V(\eps
x)u_\eps^2 \phi_\eps -\int_{H_\eps} g_\eps(x,u_\eps)u_\eps
\phi_\eps =\int_{H_\eps} \l_\eps\cdot h_\eps(x) u_\eps^2
\phi_\eps.
\]
Therefore, by \ref{en:fg}, if $\d>0$ is sufficiently small, there
exists $C_\d>0$, such that
\[
\int_{H'_\eps} |\n u_\eps|^2 +\int_{H'_\eps} \left(V(\eps
x)-\frac{\a_1}{2}-\d \right) u_\eps^2 \le
O(\eps)+C_\d\int_{H_\eps} u_\eps^{p+1},
\]
and so the conclusion follows by (\ref{eq:L2H1}).
\end{proof}

We consider a sequence $\e_k\to 0$, that we still denote by $\e$.

\begin{proposition}\label{pr:uuu} There exist $n\in \N$, $\bar c>0$ and, for all $i=1,\ldots,n$,
there exist $y_\eps^i\in B_2^\eps \cap H_\eps$, $\bar y_i \in B_2$
and $u_i\in \H \setminus \{0\}$ such that
\begin{align*}
&\eps y_\eps^i\to \bar y_i, %\nonumber
\\
&|y_\eps^i -y_\eps^j|\to \infty, \ \hbox{ if }i\neq j ,%\label{eq:yy}
\\
&u_\eps (\cdot +y_\eps^i )\rightharpoonup u_i,\ \hbox{ weakly in }\H,  %\label{eq:uuiw}
\\
&\|u_i\|\ge \bar c,  %\label{eq:barc}
\\
&\| u_\eps -\sum_{i=1}^n u_i (\cdot -y_\eps^i )\|_{H^1(H_\e)} \to
0,
\end{align*}
and $u_i$ is a positive solution of
\[
-\Delta u_i+V(\bar y_i) u_i=g(\bar y_i,u_i)+\bar \l\cdot \bar y_i
u_i.
\]
\end{proposition}

\begin{proof}
We define $\tilde{u}_\e$ the even reflection of $u_\e|_{H_\e}$
with respect to $\partial H_{\e}$. Observe that $\{\tilde{u}_\e\}$
is bounded in $\H$ and does not converge to $0$ in $L^{p+1}(\RN)$
(recall Lemma \ref{le:L2LpH}). Then, by concentration-compactness
arguments (see \cite[Lemma 1.1]{lions}), there exists $y^1_\eps
\in \R^N$ such that
\[
\int_{B(y^1_\eps,1)}\tilde u_\eps^2\ge c>0.
\]
By the even symmetry of $\tilde{u}_\e$ and by Lemma \ref{le:L4c},
we can assume that $y^1_\e \in H_\e\cap B^\e_4$. Therefore there
exists $u_1\in \H\setminus\{0\}$ such that $v_\eps^1=u_\eps
(\cdot+ y_\eps^1)\rightharpoonup u_1$, weakly in $\H$.

Observe that $v_\e^1$ solves the equation:
\begin{equation*}%\label{eq:veps}
-\Delta v_\eps^1+V(\eps x+\eps y_\eps^1) v_\eps^1=g(\eps x+ \eps
y_\eps^1,v_\eps^1) +\l_\eps\cdot  h_\eps(x+y_\eps^1) v_\eps^1,
\end{equation*}
and so, passing to the limit ,  $u_1$ is a weak solution of
\[
-\Delta u_1+V(\bar y_1) u_1=g(\bar y_1,u_1)+\bar \l\cdot \bar y_1
u_1,
\]
where $\bar{y}_1=\lim_{\e \to 0}\e y_\e^1 $.
\\
Since $y_\eps^1\in H_\eps$, we have that $\bar \l\cdot \bar y_1\le
\alpha_1/2$ and so  $\bar y_1\in B_2$ (otherwise $u_1$ should be
$0$) and, by \ref{en:fg}, we easily get that there exists $c>0$
such that $c \le \|u_1\|$. Moreover, observe that
$$ \|u_\e\| \ge \|u_1\|.$$
Let us define $w_\e^1=u_\eps - u_1(\cdot - y_{\e}^1)$. We consider
two possibilities: either $\| w_\e^1\|_{H^1(H_\e)} \to 0$ or not.
In the first case the proposition should be proved taking $n=1$.
In the second case, there are still two sub-cases: either $\|
w_\e^1 \|_{L^{p+1}(H_\e)} \to 0$ or not.
\\
\
\\
{\bf Step 1: } Assume that $\| w_\e^1 \|_{L^{p+1}(H_\e)}
\nrightarrow 0$.

In such case, we can repeat the previous argument to the sequence
$\{w_\e^1\}$: we take $\tilde{w}_\e^1$ its even reflection with
respect to $\partial H_\e$, and apply \cite[Lemma 1.1]{lions};
there exists $y^2_\eps \in H_\e$ such that
\[
\int_{B(y^2_\eps,1)}(\tilde w_\eps^1)^2\ge c>0.
\]
Therefore, as above, there exists $u_2\in \H\setminus\{0\}$ such
that $v_\eps^2=w_\eps^1 (\cdot+ y_\eps^2)\rightharpoonup u_2$,
weakly in $\H$. Moreover, $|y_\e^1 -y_\e^2 | \to +\infty$ and $\e
y_\e^2 \to \bar{y}_2 \in B_2$, and
\[
-\Delta u_2+V(\bar y_2) u_2=g(\bar y_2,u_2)+\bar \l\cdot \bar y_2
u_2,
\]
and $\|u_2\|\ge c>0$. Moreover, by weak convergence,
$$\|u_\e\|^2 \ge \|u_1\|^2 + \|u_2\|^2.$$
Let us define $w_\e^2:=w_\e^1 - u_2(\cdot - y_{\e}^2)=u_\e -
u_1(\cdot - y_{\e}^1) - u_2(\cdot - y_{\e}^2)$. Again, if $\|
w_\e^2\|_{H^1(H_\e)}\to 0$, the proof is completed for $n=2$.

Suppose now that $\| w_\e^2\|_{H^1(H_\e)}\nrightarrow 0$, $\|
w_\e^2\|_{L^{p+1}(H_\e)} \nrightarrow 0$. In such case we can
repeat the argument again.

Observe that we would finish in a finite number of steps,
concluding the proof.

The only possibility missing in our study is the following:
\begin{equation} \label{no} \mbox{ at a certain step } j,\ \ \| w_\e^j\|_{H^1(H_\e)}\nrightarrow 0, \mbox{ and } \|
w_\e^j\|_{L^{p+1}(H_\e)} \to 0,
\end{equation}
where $w_\e^j=u_\e - \sum_{k=1}^j u_k(\cdot - y_{\e}^k)$.
\\
\
\\
{\bf Step 2: } The assertion \eqref{no} does not hold.

Suppose by contradiction \eqref{no}. Let us define
$$
H_\e^1 = \left\{x\in \RN \mid \bar \l  \cdot x \le
\frac{a_2}{\eps}\right\},
$$
where $\frac{\a_1}{2}< a_1<\frac{2\a_1}{3}$. We claim that
\begin{equation}\label{eq:vu1p}
\| w_\e^j \|_{L^{p+1}(H_\e^1)} \nrightarrow 0.
\end{equation}
By \eqref{no} there exists $\d>0$ such that
\begin{equation}\label{eq:vud}
\| u_\eps \|^2_{H^1(H_\e)} \ge \sum_{k=1}^j  \| u_k(\cdot -
y_\e^k) \|^2_{H^1(H_\e)} +\d.
\end{equation}

Let us fix $R>0$ large enough and choose a cut-off function $\phi$
satisfying the following:
$$ \left \{ \begin{array}{ll} \phi= 0  & \mbox{ in } \left (\cup_{k=1}^j B(y_\e^k, R) \right ) \cup (H_\e^1)^c, \\
\phi =1   & \mbox{ in } H_\e \setminus \left (\cup_{k=1}^j
B(y_\e^j, 2R) \right ), \\0 \le \phi \le 1, & \\ |\nabla \phi |
\le C/R. &
\end{array} \right.$$
We multiply \eqref{eq:lambda} by $\phi u_\e$ and integrate to
obtain:
$$\irn \phi |\n u_\e|^2 + u_\e (\n u_\e \cdot \n \phi) + V(\e x)
\phi u_\e^2 = \irn g_\e(x,u_\e) \phi u_\e^2 + \bar{\l} \cdot
h_\e(x) \phi u_\e^2.$$ Therefore, by using \ref{en:fg} and the
properties of the cut-off $\phi$ we get:
\begin{equation}\label{eq:palle}
\int_{H_\e \setminus \left ( \cup_{k=1}^j B(y_\e^j, 2R)\right
)}\left (|\n u_\e|^2 + c u_\e^2 \right ) - \frac C R \le C
\int_{H_\e^1 \setminus \left ( \cup_{k=1}^j B(y_\e^j, R)\right )}
u_\e^{p+1}.
\end{equation}
Observe moreover that by regularity arguments $u_\e(\cdot +
y_\e^k) \to u_k$ in $H^1_{loc}$. Then \eqref{eq:vud} implies that
the left hand term in (\ref{eq:palle}) is bounded from below: this
finishes the proof of \eqref{eq:vu1p}.

\medskip Then, we can repeat the whole procedure: there exists
$y_\e^{j+1} \in H^1_\e$ such that $u_\e(\cdot + y_\e^{j+1})
\rightharpoonup u_{j+1}$. Define $w_\e^{j+1} = w_\e^j -
u_{j+1}(\cdot - y_\e^{j+1})$. Observe that since $\|
w_\e^j\|_{L^{p+1}(H_\e)} \to 0$, we have that $\dist(y_\e^{j+1},
H_\e) \to +\infty$.

Now we go on as above, replacing $H_\e$ with $H_\e^1$. If for
certain $j' \ge j+1$ we have:
\begin{equation*}% \label{no2}
\| w_\e^{j'}\|_{H^1(H^1_\e)}\nrightarrow 0, \mbox{ and } \|
w_\e^{j'}\|_{L^{p+1}(H^1_\e)} \to 0,
\end{equation*}
we argue again as in the beginning of Step 2 to deduce that $\|
w_\e^{j'}\|_{L^{p+1}(H^2_\e)} \nrightarrow 0$, where
$$
H_\e^2 = \left\{x\in \RN \mid \bar \l  \cdot x \le
\frac{a_2}{\eps}\right\},
$$
with $a_1<a_2<\frac{2\a_1}{3}$.

In so doing we can again continue our argument, eventually
introducing
$$
H_\e^l = \left\{x\in \RN \mid \bar \l  \cdot x \le
\frac{a_l}{\eps}\right\},
$$
with $a_{l-1}<a_l<\frac{2\a_1}{3}$.

Since all limit solutions $u_k$ are bounded from below in norm, we
end in a finite number $n$ of steps. Therefore, we obtain
\begin{align*}
& y_\eps^k \in H_\e \ \ \forall  k=1, \dots j,
\\
&\dist (y_\eps^k, H_\e) \to \infty, \ \ \forall  k= j+1, \dots n,  ,%\label{eq:yy}
\\
&\| u_\eps -\sum_{k=1}^n u_k (\cdot -y_\eps^k )\|_{H^1(H^q_\e)}
\to 0, \mbox {for a suitable } q.
\end{align*}
This implies that
\[
\|w_\e^j\|_{H^1(H_\e)} \le \| u_\eps -\sum_{k=1}^n u_k (\cdot
-y_\eps^k )\|_{H^1(H_\e)} +o_\e(1)=o_\e(1)
\]
but this is in contradiction with $\|w_\e^j\|_{H^1(H_\e)}
\nrightarrow 0$ assumed in \eqref{no}.

\end{proof}

Our arguments distinguish two possible situations. Let us consider
each of them separately.
\\
\
\\
{\bf Case 2a)}: $\bar \l \cdot \bar y_i \ge 0$, for all
$i=1,\ldots,n$.
\\
Since $\b_\eps(u_\eps)=0$, we have that
\[
0=\int_{H_\eps}\l_\eps \cdot h_\eps(x) u_\eps^2
+\int_{(H_\eps)^c}\l_\eps \cdot h_\eps(x) u_\eps^2.
\]
By Proposition \ref{pr:uuu} and since  $\bar \l \cdot \bar y_i \ge
0$, for all $i=1,\ldots,n$, we know that
\[
\int_{H_\eps}\l_\eps \cdot h_\eps(x) u_\eps^2 \to \sum_{i=1}^n
\bar \l \cdot \bar y_i \irn u_i^2 \ge 0,
\]
whereas $\l_\eps \cdot h_\eps(x) \ge \frac{\alpha_1}{2 \eps}$ in
$B_3^\e \setminus H_\eps$. Therefore we have
\begin{equation*}%\label{eq:ly=0}
\bar \l \cdot \bar y_i = 0, \ \hbox{ for all }i=1,\ldots,n,
\end{equation*}
and
\begin{equation}\label{pora cocca}
\frac{\alpha_1}{2 \eps} \int_{B_3^\e \setminus H_\eps} u_\eps^2
\le \int_{(H_\eps)^c} \l_\e \cdot h_\e(x) u_\eps^2 \to 0, \ \hbox{
as }\eps \to 0.
\end{equation}
With that information in hand, let us estimate the energy
$\tilde{I}_\e(u_\e)$:
\begin{align*}
\tilde I_\eps (u_\eps) &= \int_{B_2^\eps \cap H_\e} \left[\frac 12
\big(|\n u_\eps|^2 +V(\eps x) u_\eps^2\big)
-G_\eps(x,u_\eps)\right] \nonumber
\\
&\quad +\int_{B_2^\eps \setminus H_\e }  \left[\frac 12 \big(|\n
u_\eps|^2 +V(\eps x) u_\eps^2\big) -G_\eps(x,u_\eps)\right]
\nonumber
\\
&\quad + \int_{(B_2^\eps)^c}  \left[\frac 12 \big(|\n u_\eps|^2
+V(\eps x) u_\eps^2\big) -G_\eps(x,u_\eps)\right].
%\label{eq:sommadiI}
\end{align*}

By Proposition \ref{pr:uuu}, we have that
\begin{equation*}%\label{eq:B3H2}
\int_{ B_2^\eps \cap H_\e}  \left[\frac 12 \big(|\n u_\eps|^2
+V(\eps x) u_\eps^2\big) -G_\eps(x,u_\eps)\right] = \sum_{i=1}^n
J_{\bar y_i}(u_i) +o_\eps(1).
\end{equation*}
Moreover,  since $\{u_\eps\}_{\eps>0}$ is a bounded sequence in
$\H$ and so also in $L^{s}(\RN)$ (for a certain $s>p+1$), we can
use interpolation and \eqref{pora cocca} to get
\begin{equation*}
\int_{B_3^\e \setminus H_\eps} u_\eps^{p+1} \to 0, \ \hbox{ as
}\eps \to 0.
\end{equation*}
Then, we obtain:
\begin{equation*}%\label{eq:sommal3c0}
\int_{B_2^\e \setminus H_\eps}  \left[\frac 12 \big(|\n u_\eps|^2
+V(\eps x) u_\eps^2\big) -G_\eps(x,u_\eps)\right] \ge o_\e(1).
\end{equation*}

Finally, by the definition of $G_\e(x,u)$, we have:
\begin{equation*}%\label{eq:sommal3c}
\int_{(B_2^\eps)^c}  \left[\frac 12 \big(|\n u_\eps|^2 +V(\eps x)
u_\eps^2\big) -G_\eps(x,u_\eps)\right] \ge 0.
\end{equation*}

So, we get the estimate:
\begin{equation*}%\label{eq:sommaJ}
\lim_{\eps\to 0} b_\eps=\lim_{\eps\to 0} \tilde I_\eps (u_\eps)
\ge \sum_{i=1}^n J_{\bar y_i}(u_i).
\end{equation*}

Reasoning as in Case 1, we easily conclude whenever $n>1$.
Moreover, if $n=1$,
$$ 0 = \e \int_{B_3^\e} \pi_E(x) u_\e^2(x)= \e  \int_{B_3^\e \cap H_\e} \pi_E(x) u_\e^2(x) + \e \int_{B_3^\e \setminus H_\e} \pi_E(x) u_\e^2(x).$$
By \eqref{pora cocca}, the second right term of the last
expression tends to $0$. By arguing as in Case 1, we conclude:
$$  \e \int_{B_3^\e \cap H_\e} \pi_E(x) u_\e^2(x) \to
\pi_E(\bar{y}_1) \irn u_1^2.$$ Then $\bar{y}_1 \in E^{\perp}$, and
we conclude
$$ J_{\bar y_1}(u_1) \ge m_{V(\bar{y}_1)} \ge m.$$
\\
\
\\
{\bf Case 2b)}: there exists at least an $i=1,\ldots,n$ such that
$\bar \l \cdot \bar y_i<0$.

Without lost of generality, we can assume that $\bar \l \cdot \bar
y_1<0$. Let $s>0$ such that $B(\bar y_1,3s) \subset B_3$, with
$\bar y_i \notin B(\bar y_1,3s)$ for all $\bar y_i\neq \bar y_1$,
and such that $\bar \l \cdot x<0$, for all $x\in B(\bar y_1,3s) $.
We define $B_s^\e=\eps^{-1}B(\bar y_1,s)$ and
$B_{2s}^\e=\eps^{-1}B(\bar y_1,2s)$. By Proposition \ref{pr:uuu},
there exists $c>0$ such that
\begin{equation}\label{eq:ber}
\int_{B_s^\e}u_\eps^2 \ge c>0.
\end{equation}
Let $\phi_\eps$ be a smooth function such that
\[
\phi_\eps(x)=\left\{
\begin{array}{ll}
1& \hbox{if } x\in B_s^\e,
\\
0& \hbox{if }x \in (B_{2s}^\e)^c,
\end{array}
\right.
\]
with $0\le \phi_\eps \le 1$ and $|\n \phi_\eps| \le C \eps$.
Repeating the arguments of the proof of Lemma \ref{le:Oeps}, we
multiply \eqref{eq:lambda} by $(\de_{\tilde \l_\eps} u_\eps
)\phi_\eps$, where $\tilde \l_\eps=\l_\eps/|\l_\eps|$. We have
\begin{align}
&\int_{B_{2s}^\e}\big[\n u_\eps\cdot \n(\de_{\tilde \l_\eps}
u_\eps )\big]\phi_\eps +\int_{B_{2s}^\e\setminus B_s^\e}(\n
u_\eps\cdot \n \phi_\eps )\de_{\tilde \l_\eps} u_\eps \nonumber \\
&\qquad +\int_{B_{2s}^\e}V(\eps x)u_\eps(\de_{\tilde \l_\eps}
u_\eps )\phi_\eps -\int_{B_{2s}^\e} g_\eps( x, u_\eps)(\de_{\tilde
\l_\eps} u_\eps )\phi_\eps \nonumber
\\
&\quad =\int_{B_{2s}^\e}(\l_{\eps} \cdot h_\eps(x))
u_\eps(\de_{\tilde \l_\eps} u_\eps )\phi_\eps . \label{eq:**bis}
\end{align}
Let us evaluate each term of the previous equality. Since
\[
\| u_\eps \|_{H^1(B_{2s}^\e\setminus B_s^\e)}\to 0,
\]
we have
\begin{align}
\label{eq:lnubis} \int_{B_{2s}^\e}\big[\n u_\eps\cdot
\n(\de_{\tilde \l_\eps} u_\eps )\big]\phi_\eps
&=-\frac{1}{2}\int_{B_{2s}^\e\setminus B_s^\e}|\n u_\eps|^2
\de_{\tilde \l_\eps} \phi_\eps =o(\eps).
\\
\label{eq:lnchibis} \int_{B_{2s}^\e\setminus B_s^\e}(\n
u_\eps\cdot \n \phi_\eps )\de_{\tilde \l_\eps} u_\eps &=o(\eps).
\end{align}
Analogously, we have
\begin{align}
\int_{B_{2s}^\e}  V(\eps x)u_\eps(\de_{\tilde \l_\eps} u_\eps
)\phi_\eps &=-\frac{\eps}2 \int_{B_{2s}^\e} (\de_{\tilde \l_\eps}
V(\eps x))u_\eps^2 \phi_\eps -\frac 12\int_{B_{2s}^\e\setminus
B_s^\e} V(\eps x) u_\eps^2 \de_{\tilde \l_\eps} \phi_\eps
\nonumber
\\
&=-\frac{\eps}2 \int_{B_{2s}^\e} (\de_{\tilde \l_\eps} V(\eps
x))u_\eps^2 \phi_\eps +o(\eps).   \label{eq:lvbis}
\end{align}
Observe that $\de_{\tilde \l_\eps} \chi (\eps x)\ge 0$ for all
$x\in B_{2s}^\e$; this is the key point of our estimates in this
case. Then, by \ref{en:ft2} we get that
\begin{align}
\int_{B_{2s}^\e} \! g_\eps( x,u_\eps)(\de_{\tilde \l_\eps} u_\eps
) \phi_\eps &=-\eps\int_{B_{2s}^\e} \!\! (F(u_\eps)-\tilde
F(u_\eps))(\de_{\tilde \l_\eps} \chi (\eps x)) \phi_\eps
-\int_{B_{2s}^\e\setminus B_s^\e}  \!\!\!G_\eps (x,u_\eps)
\de_{\tilde \l_\eps} \phi_\eps \nonumber
\\
&\le o(\eps).   \label{eq:lgbis}
\end{align}
Finally
\begin{equation}
\int_{B_{2s}^\e}  (\l_{\eps} \cdot h_\eps(x)) u_\eps(\de_{\tilde
\l_\eps} u_\eps ) \phi_\eps =
-\frac{1}2|\l_{\eps}|\int_{B_{2s}^\e} u^2_\eps \phi_\eps +o(\eps).
\label{eq:lhbis}
\end{equation}
Therefore,  by (\ref{eq:ber})--(\ref{eq:lhbis}), we obtain the
inequality:
\begin{align*}
c(|\bar \l| +o_\eps(1)) &\le \frac{|\l_{\eps}|}\eps
\int_{B_{2s}^\e} u^2_\eps \phi_\eps \le  \int_{B_{2s}^\e}
(\de_{\tilde \l_\eps} V(\eps x))u_\eps^2 \phi_\eps +o_\eps(1)
\\
&\le C \max_{x\in B_3} |\n V(x)| +o_\eps(1).
\end{align*}
We can choose $B_3$ sufficiently small such that, for a suitable
$\bar \d>0$, we have that $|\bar \l|<\bar \d$ and
\[
B_3^\eps \subset H_\eps.
\]

%To be precise, we could choose
%\begin{equation}\label{eq:r3}
%R_3\le \frac{\bar c }{2 \bar C  \max_{x\in B_3} |\n V(x)|},
%\end{equation}
%where $\bar c$ is a uniform lower bound of $\irn U^2$, with $U$ a
%solution of
%\[
%-\Delta U+m U=g(\bar y,U)
%\]
%for suitable $m$ and $\bar y$ varying in fixed bounded sets, and
%$\bar C$ is a uniform upper bound of $\irn u_\eps^2$.
%\medskip

Now we can estimate $\tilde{I}_\e(u_\e)$ in the following way:

\begin{align}
\tilde I_\eps (u_\eps) &= \int_{B_2^\eps } \left[\frac 12 \big(|\n
u_\eps|^2 +V(\eps x) u_\eps^2\big) -G_\eps(x,u_\eps)\right]
\nonumber
\\
&\quad +\int_{(B_2^\eps)^c}  \left[\frac 12 \big(|\n u_\eps|^2
+V(\eps x) u_\eps^2\big) -G_\eps(x,u_\eps)\right]. \nonumber
\label{eq:sommadiI-bis}
\end{align}

Since $B_3^\e \subset H_\e$, we can apply Proposition \ref{pr:uuu}
to obtain:
\begin{equation*}%\label{eq:B3H2}
\int_{B_2^\eps}  \left[\frac 12 \big(|\n u_\eps|^2 +V(\eps x)
u_\eps^2\big) -G_\eps(x,u_\eps)\right] = \sum_{i=1}^n J_{\bar
y_i}(u_i) +o_\eps(1).
\end{equation*}

Moreover, by the definition of $G_\e(x,u)$, we have:
\[ \int_{(B_2^\eps)^c}  \left[\frac 12 \big(|\n u_\eps|^2
+V(\eps x) u_\eps^2\big) -G_\eps(x,u_\eps)\right]\ge 0.\]

Then, we conclude that:
$$ \tilde I_\eps (u_\eps) \ge \sum_{i=1}^n J_{\bar
y_i}(u_i) +o_\eps(1).$$

As in Case 1, we conclude easily if $n>1$. Assume now that $n=1$;
since $B_3^\e \subset H_\e$, we can argue as in Case 1 to obtain:
$$ 0 = \int_{B_3^\e} \pi_E(x) u_\e^2(x) \to \pi_E(\bar{y}_1) \irn u_1^2.$$
But this is in contradiction with the hypothesis of Case 2b),
namely, $\bar{\l} \cdot \bar{y}_1<0$.

\end{proof}

\section{Asymptotic behavior} \label{Asymptotic behavior}

In this section we will study the asymptotic behavior of the
solution obtained in Section \ref{The min-max argument}. As a
consequence, $u_{\e}$ will be actually a solution of
\eqref{eq:11}: in this way we conclude the proof of Theorem
\ref{teo}.

Let us define $u_{\e}$ the critical point of $\tilde{I}_\e$ at
level $m_{\e}$, that is,
\begin{equation} \label{eq:nose}
-\Delta u_{\e} + V(\e x) u_\e ={g}_{\e}(x,u_\e).
\end{equation}
Moreover, Propositions \ref{estima} and \ref{fund} imply that
$\tilde{I}_{\e}(u_{\e}) \to m$.

The following result gives a description of the behavior of
$u_{\e}$ as $\e \to 0$:

\begin{proposition} \label{hola} Given a sequence $\e=\e_{j} \to 0$, there exists a
subsequence (still denoted by $\e_j$) and a sequence of points
$y_{\e_j} \in \R^N$ such that:

\begin{itemize}
\item $\e_j y_{\e_j} \to 0$. \item $\| u_{\e_j} - U(\cdot -
y_{\e_j})\| \to 0$,
\end{itemize}
where $U \in \mathcal{S}$ (see (\ref{eq:S})).
\end{proposition}

\begin{proof}
For the sake of clarity, let us write $\e=\e_j$. Our first tool is
again Proposition 4.2 of \cite{jt-cv}; there exist $l \in \N$,
sequences $\{y_{\e}^k\} \subset \R^N$, $\bar y_k \in B_2$, $U_k
\in H^1(\R^N)\setminus\{0\}$ ($k=1, \dots l$) such that:

\begin{itemize}

\item $|y_{\e}^k - y_{\e}^{k'}| \to + \infty$ if $k \neq k'$,

\item $\e y_{\e}^k \to \bar y_k$,

\item $\displaystyle \left \| u_{\e} - \sum_{k=1}^l U_k(\cdot -
y_{\e}^k) \right \| \to 0,$

\item $J_{\bar y_k}'(U_k)=0$,

\item $\tilde{I}_{\e}(u_{\e}) \to \sum_{k=1}^l J_{\bar y_k}(U_k)$.

\end{itemize}
For the definition of $J_{\bar y_k}$ see \eqref{def:J}. Observe
that $J_{\bar y_k}(U_k) \ge m_{V(\bar y_k)}$ since $J_{\bar y_k}
\ge \Phi_{V(\bar y_k)}$. Moreover, Lemma \ref{le:cm} implies that
$ m_{V(\bar y_k)} \ge m - \delta$ for any $\bar y_k \in B_2$,
where $\delta>0$ can be taken arbitrary small by appropriately
shrinking $B_2$: this implies that $l=1$. So, the only thing that
remains to be proved is that $\bar y_1=0$.

Our argument here has been used already in the previous section,
so we will be sketchy. By regularity arguments, $\{u_\e\} \subset
H^2(\R^N)$ and is bounded. Choose $r>0$ and $\phi_{\e}$ a cut-off
function so that $\phi_{\e}(x)=1$ in $B(y_\e^1,r \e^{-1})$ and
$\phi_{\e}(x)=0$ if $x \in B(y_\e^1,2r\e^{-1} )^c$, with $|\nabla
\phi_{\e}| \le C \e$. By multiplying \eqref{eq:nose} by
$\phi_{\e}(x)
\partial_{\nu} u_{\e}$ and integrating, we obtain:
\begin{equation} \label{esteb0} \frac{1}{2} \e  \int_{B(y_\e^1,\e^{-1} r)} \partial_{\nu} V(\e x) u_{\e}^2(x) \,  -
\e  \int_{B(y_\e^1,\e^{-1} r)} \partial_{\nu} \chi(\e x)
[F(u_{\e}(x))- \tilde{F}(u_{\e}(x))]\,  = o(\e).
\end{equation}

If $\chi$ is $C^1(B(\bar y_1,r))$, we divide by $\e$ and pass to
the limit to obtain:
\begin{equation} \label{esteb} \frac{1}{2} \partial_{\nu}V(\bar y_1) \int_{\R^N} U_1^2(x) \,  -
\partial_{\nu} \chi (\bar y_1) \int_{\R^N} [F(U_1(x))- \tilde{F}(U_1(x))]\,  = 0.
\end{equation}

We consider three different cases:

\medskip

{\bf Case 1:} $\bar y_1 \in B_1$.

Take $r>0$ so that $B(\bar y_1,2r) \subset B_1$. By \eqref{esteb},
we get that $\partial_{\nu} V(\bar y_1)=0$. Since $\nu$ is
arbitrary, $\bar y_1$ is a critical point of $V$ in $B_1$, and
therefore $\bar y_1=0$.

\medskip

{\bf Case 2:} $\bar y_1 \in B_2 \setminus \overline{B_1} $.

In this case we will arrive to a contradiction. Take $r>0$ so that
$B(\bar y_1,2r) \subset B_2 \setminus B_1$ and $\nu =
\frac{1}{|\bar y_1|}\bar y_1$. By the definition of $\chi$ (see
\eqref{def:chi}), $
\partial_{\nu} \chi (\bar y_1)= -1/(R_2-R_1)$.

We now use the Pohozaev identity for $U_1$ to get:
\begin{multline*}
\int_{\R^N} \Big ( \frac{N-2}{2} |\nabla U_1|^2 +
\frac{N}{2}V(\bar y_1) U_1^2 \Big) = N \int_{\R^N} \chi(\bar y_1)
F(U_1) + (1- \chi(\bar y_1)) \tilde{F}(U_1)
\\
\qquad \le a \frac{N}{2} \int_{\R^N} U_1^2(x)\,  + N\chi(\bar y_1)
\int_{\R^N} [F(U_1(x))- \tilde{F}(U_1(x))]
 \end{multline*}
 and so
 \[
 c \int_{\R^N} U_1^2 \le \int_{\R^N} [F(U_1(x))- \tilde{F}(U_1(x))]\, .
\]
So, it suffices to take $R_2-R_1$ smaller, if necessary, to get a
contradiction with \eqref{esteb}.

\medskip

{\bf Case 3:} $\bar y_1 \in \partial B_1$.

Also in this case we obtain a contradiction. Indeed, observe that
here $\chi(\bar y_1)=1$, and so $U^1$ is a solution of:
$$ -\Delta U_1 + V(\bar y_1)U _1= f(U_1). $$

Since $J_{\bar y_1}(U_1)=\Phi_{V(\bar y_1)}(U_1)=m$, Lemma
\ref{le:cm} implies that $V(\bar y_1)=1$. By \eqref{elenuzza},
then, there exists $\tau \in \R^N$ tangent to $\partial B_1$ at
$\bar y_1$ such that $\partial_{\tau} V(\bar y_1) \neq 0$.

We now argue as above, with the exception that here $\chi$ is not
$C^1$. However, it is a Lipschitz map so that \eqref{esteb0}
holds: let us choose $r< R_2-R_1$ and $\nu= \tau$. Now we can
write:
\begin{align*}
&\left | \int_{B(y_\e^1, r/\e)} \partial_{\tau} \chi(\e x)
[F(u_{\e}(x))- \tilde{F}(u_{\e}(x))]\,   \right |
\\
&\quad\le \frac{1}{R_2-R_1}\int_{B(0,r/\sqrt \e)}  \left[\frac{|x
\cdot \tau |}{|x+y_\e^1|} +\frac{|y_\e^1 \cdot \tau |}{|x+y_\e^1|}
\right] [F(u_{\e}(x+y_\e^1))- \tilde{F}(u_{\e}(x+y_\e^1))]
\\
&\qquad+\frac{1}{R_2-R_1}\int_{r/\sqrt \e\le |x|\le r/\e}
\frac{|(x+y_\e^1 )\cdot \tau |}{|x+y_\e^1|}  [F(u_{\e}(x+y_\e^1))-
\tilde{F}(u_{\e}(x+y_\e^1))] \, \to 0.
\end{align*}
In the above limit we have used again the dominated convergence
theorem and the strong convergence of $u_\e(\cdot + y_{\e})^1$.
Then, we can divide by $\e$ and pass to the limit in
\eqref{esteb0} to get:
$$
\frac{1}{2} \partial_{\tau} V(\bar y_1) \int_{\R^N} U_1^2(x) \,=0,
$$
a contradiction.

\end{proof}

\begin{proof}[Proof of Theorem \ref{teo}] It suffices to show that
$u_{\e}$ is a solution of \eqref{eq:Sch}. Let us show that indeed
$u_{\e}(x) \to 0$ as $\e \to 0$ uniformly in $x \in (B_1^{\e})^c$.
By Proposition \ref{hola} we obtain:
$$ \|u_{\e}\|_{H^1((B_0^{\e})^c)} \le \|u_{\e} - U(\cdot - y_{\e})\| + \| U(\cdot -
y_{\e})\|_{H^1((B_0^{\e})^c)} \to 0,$$ as $\e \to 0$. By using
local $L^{\infty}$ regularity of $u_\e$, given by standard
bootstrap arguments, we obtain that for any $x \in (B_1^{\e})^c$,
$$ \|u_{\e}\|_{L^{\infty}(B(x,1))} \le C
\|u_{\e}\|_{H^1(B(x,2))}  \le C \|u_{\e}\|_{H^1((B_0^{\e})^c)}\to
0.$$ This concludes the proof.

\end{proof}

\section{Appendix} \label{appendix}

In this section we prove Proposition \ref{sard}, that has been
used in the definition of the truncation (see \eqref{elenuzza}).
Moreover, we will discuss some possible extensions of our results.

\begin{proposition} \label{sard} Let $V: B(0,R)\subset \R^N \to \R$ be a $C^{N-1}$ function
with a unique critical point at $0$, and assume that $V(0)=1$.
Then, the following assertion is satisfied for almost every $R\in
(0, R_0)$:
\begin{equation} \label{elenuzza2} \forall x \in \partial
B(0,R)\mbox{ with }  V(x)=1, \partial_{\tau}V(x) \neq 0, \mbox{
where } \tau \mbox{ is tangent to } \partial B(0,R) \mbox{ at }x.
\end{equation}
\end{proposition}

\begin{proof}

The proof is an easy application of the Sard lemma. Given $\delta
\in (0, R_0)$, let us define the annulus $A=A(0; \delta, R_0)$.
Let us consider the set:
$$M = \{ x \in A \mid V(x)=1\}.$$

If $M$ is empty, we are done. Otherwise, since $V$ has no critical
points in $A$, the implicit function theorem implies that $M$ is a
$N-1$ dimensional manifold with $C^{N-1}$ regularity and a finite
number of connected components; then, we can decompose $M=
\cup_{i=1}^n M_i$, where $M_i$ are connected.

Let us define the maps:
$$\psi_i: M_i \to \R, \ \psi(x)= |x|.$$
Since $M_i$ is a $C^{N-1}$ manifold, we can apply Sard lemma: if
we denote by $S_i\subset (\delta, R_0)$ the set of critical values
of $\psi_i$, then $S_i$ has $0$ Lebesgue measure in $\R$. Define
$S = \cup_{i=1}^n S_i$. It can be checked that for any $R \in
(\delta, R_0) \setminus S$, \eqref{elenuzza2} holds.

Now, it suffices to take $\delta_n \to 0$ and $S^n$ the
corresponding set of critical values. Clearly, $\cup_{n \in \N}
S^n$ has also $0$ Lebesgue measure, and this finishes the proof.

\end{proof}

Now we discuss some slight extensions of our result. As we shall
see, a couple of hypotheses of Theorem \ref{teo} can be relaxed.
However, we have preferred to keep Theorem \ref{teo} as it is,
because in this form the proof becomes more direct and clear. So,
let us now discuss those extensions of our results, as well as the
modifications needed in the proofs.

\bigskip {\bf 1. Condition \ref{f0}.} The $C^1$ regularity of $f(u)$ implies
that all ground states of \eqref{eq:lp} are radially symmetric
(actually, $C^{0,1}$ regularity suffices). However, this is not
really necessary in our arguments. Indeed, in \cite{jb-arma} it is
proved that the set $\mathcal{S}$ is compact, up to translations,
even for continuous $f(u)$. So, in Section \ref{The min-max
argument} it suffices to take $\gamma(t)$ related to $U \in
\mathcal{S}$ such that:
$$ \int_{\R^N} U(x) x \,  =0.$$
Moreover, we cannot use compact embeddings of $H^1_r(\R^N)$ in the
proof of Lemma \ref{le:cm}: the proof of that lemma must be
finished by making use of concentration-compactness arguments.

\bigskip {\bf 2. Condition \ref{V0}.} The lower bound on $V$ is
strictly necessary in our arguments; the upper bound, though it
has been imposed to make many computations, have a clearer form.
Indeed, condition \ref{V0} can be replaced with:

\begin{enumerate}[label=(V0'), ref=(V0')]

\item \label{V0bis} $0<\a_1 \le V(x)  \quad x \in \RN$.

\end{enumerate}

In such case, some technical work is in order. First, we need to
consider the norm:
$$\| u \|_V = \left ( \int_{\R^N} |\n u|^2 + V(x) u^2
\right)^{1/2},$$ and the Hilbert space $H_V$ of functions $u \in
\H$ such that $\|u\|_V$ is finite. Then, it is not obvious that
the solutions $U \in \mathcal{S}$ belong to $H_V$. Therefore, we
need to define a cut-off function $\psi_\e$ such that $\psi_\e = 1
$ in $B_2^\e$, $\psi_\e =0 $ in $(B_3^\e)^c$ and $|\n \psi_\e| \le
C \e$.

The cone $\Ce$ defined in \eqref{def:cono} is to be replaced with
\begin{equation*}% \label{def:cono-bis}
\overline{\Ce}=\left\{ \psi_\e \gamma_t(\cdot-\xi)\;\vline\;
t\in[0,1], \xi\in \overline{B_0^\eps}\cap E \right\}.
\end{equation*}

The estimates that would follow become more technical, but no new
ideas are required.

\end{document}